\documentclass{elsart}

\usepackage{amssymb}
\usepackage{color}
\usepackage{graphicx}
\usepackage{subfigure}
\usepackage{amsmath}
\usepackage{amsfonts}

\usepackage[colorlinks=true]{hyperref}

\usepackage[title]{appendix}
\usepackage{latexsym}
\usepackage{multirow}
\usepackage{epstopdf}
\usepackage{epsfig}
\usepackage{psfrag}
\usepackage{paralist}

\usepackage{geometry}
\newtheorem{example}{Example}[section]
\newtheorem{lemma}{Lemma}[section]
\newtheorem{remark}{Remark}[section]
\newtheorem{theorem}{Theorem}[section]
\numberwithin{equation}{section}
\numberwithin{figure}{section}
\numberwithin{table}{section}
\numberwithin{theorem}{section}

\newenvironment{proof}[1][Proof]{\begin{trivlist}
\item[\hskip \labelsep {\bfseries #1}]}{\end{trivlist}}
\renewcommand{\qed}{\hfill \nobreak \ifvmode \relax \else
      \ifdim\lastskip<1.5em \hskip-\lastskip
      \hskip1.5em plus0em minus0.5em \fi \nobreak
      \vrule height0.75em width0.5em depth0.25em\fi}

\renewcommand{\vec}[1]{\mbox{\boldmath \small $#1$}}

\begin{document}

\begin{frontmatter}

\title{A stochastic Galerkin method for general  system of quasilinear hyperbolic conservation laws with uncertainty}

\author{Kailiang Wu},
\ead{wukl@pku.edu.cn}
\author[label2]{Huazhong Tang}
\ead{hztang@math.pku.edu.cn}
\thanks[label2]{Corresponding author. Tel:~+86-10-62757018;
Fax:~+86-10-62751801.}
\address{HEDPS, CAPT \& LMAM, School of Mathematical Sciences, Peking University, Beijing 100871, P.R. China}
\author{Dongbin Xiu}
\ead{dongbin.xiu@utah.edu}
\address{Department of Mathematics, and Scientific Computing and Imaging
Institute, University of Utah, Salt Lake City, UT 84112}

\date{\today{}}

\maketitle

\begin{abstract}
This paper is concerned with generalized polynomial chaos (gPC)
approximation for a general  system of
quasilinear hyperbolic conservation laws with uncertainty.
The  one-dimensional (1D) hyperbolic system is
first symmetrized with the aid of left eigenvector matrix of the
Jacobian matrix. Stochastic Galerkin method is then applied to derive
the equations for the gPC expansion coefficients.
The resulting deterministic gPC Galerkin system is proved to be
symmetrically  hyperbolic. This important property then allows one to
use a variety of numerical schemes for spatial and temporal
discretization. Here
a  higher-order and path-conservative finite volume WENO scheme is
adopted in space, along with a third-order total variation diminishing
Runge-Kutta method in time.
The method is further extended to two-dimensional (2D) quasilinear
hyperbolic system with uncertainty, where the symmetric hyperbolicity
of the one-dimensional system is carried over via the operator
splitting technique.
Several 1D and 2D numerical experiments
are conducted to demonstrate the accuracy and effectiveness of the proposed gPC
stochastic Galerkin method.
\end{abstract}

\end{frontmatter}

\section{Introduction}


This paper is concerned with uncertainty quantification (UQ) of general
system of quasilinear hyperbolic conservation laws.
UQ has received increasing attention in recent years and found its use
in many problems.
One of the most widely used UQ methods is generalized polynomial chaos (gPC) \cite{Xiu2002}.
As an extension of the classical polynomial chaos \cite{Ghanem1991},
gPC approximates the uncertain solutions as a (truncated) generalized
Fourier series by utilizing orthogonal polynomials, and the unknown
expansion
coefficient functions can be computed by an intrusive or a non-intrusive method.
The intrusive methods typically employ  stochastic Galerkin (SG) projection, which
results in  a larger coupled deterministic system of equations for the
gPC coefficients.
The non-intrusive methods are often of stochastic collocation (SC) type. They solve
the original problem at some sampling points of the random variables,
and then evaluate the gPC coefficients by using the polynomial
interpolation or numerical quadrature,
e.g. \cite{Reagan2003,Xiu2005,Babuska2007,Nobile2008}. For a review of
the methods, see \cite{Xiu2010}.

Although the gPC-SG method has been successfully applied to a large
variety of problems,
its applications to hyperbolic conservation laws is still quite
limited. This is mainly due to the lack of theoretical
understanding of the resulting deterministic gPC-SG system.
For the linear and scalar hyperbolic equations,
 the resulting gPC-SG systems are still hyperbolic,
 see e.g. \cite{Chen2005,Gottlieb2008,Pulch2012,Jin:JSC2015}.
 However, for a general system of quasilinear hyperbolic conservation
 laws, the resulting gPC-SG system may be not globally hyperbolic
 \cite{Despres2013}.
The lack of hyperbolicity  means that the Jacobian matrix  may contain
complex eigenvalues, which lead to ill-posedness of the initial or
boundary problem and  instability of the numerical
computations. Recently, some efforts were  made to obtain well-behaved
gPC-SG system for a system of  hyperbolic conservation laws.
Despr\'es et al. used the gPC approximation to the entropy variables
instead of the conservative variables in the Euler equations
\cite{Despres2013}, and  proved that the resulting gPC-SG system is
hyperbolic, based on the fact that the Euler equations can be reformulated
in a symmetrically hyperbolic form in term of the entropy variables.
This method, however, can not be extended to a general quasilinear
hyperbolic system without a convex entropy pair or a non-symmetrically
hyperbolic system.
Moreover,  a minimization problem needs be solved at each spatial
mesh point and time step, and thus the method is time-consuming,
especially for multi-dimensional problems.
An approach using the Roe variables was proposed for the
Euler equations in \cite{Pettersson2014}. Although effective, its
extension to general systems is also very limited, due to the Roe
linearization.
 More recently, a class of operator splitting based SG methods
were
developed for the Euler equations  \cite{Chertock:Euler} and the
Saint-Venant system  \cite{Chertock:Shallow}.
The idea is to split the underlying system into several subsystems,
and the gPC-SG method for each of these subsystems may result in
globally hyperbolic gPC-SG system.
However, such splitting is problem dependent and difficult
to extended  to a general system of quasilinear hyperbolic
conservation laws.

The gPC Galerkin solution method for a general system
of quasilinear hyperbolic conservation laws, which is still an open
problem, is discussed in this paper. The major contribution of this
paper is the development of a gPC Galerkin approach that results a
symmetrically hyperbolic
system of equations for the gPC coefficients, for any general
quasilinear conservation laws, e.g., Euler equations.
The key ingredient of the method is the symmetrization of a general 1D
hyperbolic  system via the left eigenvector matrix of its Jacobian matrix.
The symmetric form of the 1D hyperbolic system is then discretized and
approximated by the gPC Galerkin approach. It is then proven that the
resulting larger gPC-SG system is symmetrically hyperbolic.
The symmetric hyperbolicity of the gPC-SG system  is an important
property and allows one to
employ a variety of proper numerical schemes.
In this paper a fifth-order accurate, path-conservative, finite volume
WENO scheme is used in space, and a third-order accurate, total
variation diminishing, explicit Runge-Kutta method is used in time.
For multi-dimensional problems, operator splitting technique can be
employed to take advantage of the hyperbolicity of the one-dimensional
gPC-SG systems.

The rest of the paper is organized as follows: Section \ref{sec:1D} presents the gPC-SG method for the general 1D system of hyperbolic conservation laws with uncertainty, including the discretization in random space in Subsection \ref{sec:1D:random}, the spatial discretization in Subsection \ref{sec:1D:space}, and
the time discretization in Subsection \ref{sec:1D:time}.
Section \ref{sec:2D} extends  the proposed gPC-SG method to multi-dimensional case.
Section \ref{sec:experiments} conducts several 1D and 2D numerical experiments
 to demonstrate the performance and accuracy of the
proposed gPC-SG method. Concluding remarks are presented in Section \ref{sec:conclude}.

%
%




\section{One-dimensional gPC-SG method}\label{sec:1D}

This section considers the gPC-SG method of a general quasilinear hyperbolic system
\begin{equation}\label{eq:1D}
\frac{ \partial } {\partial t}  \vec U(x,t,\vec \xi)  +
\frac{ \partial } {\partial x} \vec F(\vec U(x,t,\vec \xi);\vec \xi )
= \vec 0, \quad  x \in \Omega \subseteq \mathbb{R},~  t >0,
\end{equation}
where $\vec U(x,t,\vec \xi)\in \mathbb{R}^N$ is the unknown, $\vec
F(\vec U;\vec \xi)\in \mathbb{R}^N$ is the flux function, and $\vec
  \xi \in \Theta \subset \mathbb{R}^d$ denotes the random variables
that parameterize the uncertain coefficients or initial conditions of the given problem.
The system \eqref{eq:1D} is hyperbolic if the Jacobian matrix $\vec
A(\vec U; \vec \xi) :=  \partial \vec F(\vec U;\vec \xi) / \partial
\vec U \in \mathbb{R}^{N \times N}$ is real diagonalizable for all
admissible state $\vec U \in {\cal G} $ and almost everywhere $\vec \xi \in \Theta $,
where $\cal G$ denotes the set of admissible state.

\subsection{Symmetrization}\label{sec:1D:symm}

For smooth solution, the system \eqref{eq:1D} can be equivalently rewritten in the quasilinear form
\begin{equation}\label{eq:1Dnoncon}
\frac{ \partial } {\partial t}  \vec U(x,t,\vec \xi)  + \vec A(\vec U;\vec \xi ) \frac{ \partial } {\partial x} \vec U(x,t,\vec \xi) = \vec 0.
\end{equation}
Let the invertible matrix $\vec{L}(\vec U;\vec \xi) $ be the left eigenvector matrix of the matrix $\vec A(\vec U; \vec \xi) $, then
$$
\vec A= \vec L^{-1} \vec \Lambda (\vec U,\vec \xi) \vec L,
\quad a.e. ~~\vec \xi\in \Theta,
$$
where the diagonal matrix $\vec \Lambda (\vec U,\vec \xi)  = {\mathrm{diag}} \left\{ \lambda_1(\vec U,\vec \xi), \cdots, \lambda_N(\vec U,\vec \xi) \right\}$ and
$ \lambda_{\ell}(\vec U,\vec \xi)$, $\ell=1,\cdots,N$, are  $N$ eigenvalues of $\vec A(\vec U; \vec \xi) $.
If multiplying \eqref{eq:1Dnoncon} from the left
by the positive-definite matrix $\vec A_0 (\vec U;\vec \xi):= \vec L^{\mathrm T} \vec L$,
one may obtain a symmetric system
\begin{equation}\label{eq:1Dsymmetric}
\vec A_0 (\vec U;\vec \xi) \frac{ \partial } {\partial t}  \vec U(x,t,\vec \xi)  + \vec A_1(\vec U;\vec \xi ) \frac{ \partial } {\partial x} \vec U(x,t,\vec \xi) = \vec 0,
\end{equation}
where $ \vec A_1(\vec U;\vec \xi ) :=  \vec L^{\mathrm T} \vec \Lambda \vec L $ is a real symmetric matrix and $\vec L^{\mathrm T}$ is the transpose of $\vec L$.
 The system \eqref{eq:1Dsymmetric} is the starting point of our
stochastic Galerkin method for the system \eqref{eq:1D}.

\subsection{Discretization in random space}\label{sec:1D:random}

Let $\{ \phi_i (\vec \xi) \}_{i \in \mathbb{N}\cup\{0\}}$ represent a complete set of { polynomials in $d$ variables and are orthonormal in the sense that
$$
 \int_\Theta {\phi _i}(\vec \xi){\phi _j}(\vec \xi) {\mathrm d} \mu(\vec \xi) = \delta_{ij},
$$
where  $\mu(\vec \xi)$  is the probability distribution function of $\vec \xi$, and $\delta_{ij}$ is the Kronecker symbol.}
We expect to seek a finite approximation
\begin{equation}\label{eq:gPCAPP}
\vec u_M (x,t,\vec \xi) = \sum\limits_{i = 0}^M \big({{\vec {\hat u}_i}(x,t)\big)^{\mathrm{T}} {\phi _i}(\vec \xi)} \in {\cal G},
\end{equation}
to the solution $\vec U(x,t,\vec \xi)$,
where ${\vec {\hat u}_i}(x,t)$ is row vector, $i = 0,\cdots, M$.
To derive the gPC-SG method for the system \eqref{eq:1D},
\eqref{eq:gPCAPP} is  substituted into  \eqref{eq:1Dsymmetric} and then enforce the residual to be
orthogonal to ${\mathrm {span}}\{\phi_0(\vec \xi) ,\cdots,\phi_M(\vec \xi)\}$.
The final deterministic gPC system for the expansion coefficients
is  as follows
\begin{equation}\label{eq:1Ddeterministic}
\hat{\vec A}_0 \frac{\partial }{\partial t} \hat{\vec U}(x,t) + \hat{\vec A}_1 \frac{\partial }{\partial x} \hat{\vec U} (x,t) = \vec 0,
\end{equation}
where $\hat{\vec U} := ( \vec {\hat u}_0, \cdots,  \vec {\hat u}_M
)^{\mathrm{T}} \in \mathbb{R}^{(M+1)N}$ and the coefficient matrices $\hat{\vec A}_k \in \mathbb{R}^{(M+1)N\times (M+1)N}$, $k = 0,1$, are of  the form
$$
\hat{\vec A}_k = \left( {\begin{array}{*{20}{c}}
   \hat{\vec A}_{00}^{(k)} &  \cdots  & \hat{\vec A}_{0M}^{(k)}  \\
    \vdots  & {} &  \vdots   \\
   \hat{\vec A}_{M0}^{(k)} &  \cdots  & \hat{\vec A}_{MM}^{(k)}
\end{array}} \right),
$$
here the blocks or sub-matrices are defined by
$$
\hat{\vec A}_{ij}^{(k)} = \int_\Theta  {{\phi _i}(\vec \xi){\phi _j}(\vec \xi){\vec A_k}(\vec u_M(x,t,\vec \xi);\vec \xi)}{\mathrm {d}}
{\mu(\vec \xi)}
 \in {{\mathbb R}^{N \times N}}, \quad i,j = 0, \cdots M.
$$

\begin{theorem}\label{th:symmetric}
If $\vec u_M(x,t,\vec \xi)\in {\cal G}$, then $\hat{\vec A}_0$ is real symmetric and positive definite, and $\hat{\vec A}_1$ is real symmetric, that is to say, the gPC-SG system \eqref{eq:1Ddeterministic} is symmetrically hyperbolic.
\end{theorem}
\begin{proof}
Both matrices $\hat{\vec A}_0$ and $\hat{\vec A}_1$ are real since $\vec u_M(x,t,\vec \xi)\in {\cal G}$.
Let $\tilde{\vec A}_{ij}^{(k)} \in \mathbb{R}^{N \times N}$ for $i,j=0,\cdots,M$ be the blocks of
the matrix  $\hat{\vec A}_k^{\mathrm{T}}$, then one has
\begin{align*}
\tilde{\vec A}_{ij}^{(k)} &= \big ( \hat{\vec A}_{ji}^{(k)} \big) ^{\mathrm{T}} = \left(  \int_\Theta  {{\phi _i}(\vec \xi){\phi _j}(\vec \xi){\vec A_k}(\vec u_M(x,t,\vec \xi);\vec \xi)}{\mathrm {d}}  {\mu(\vec \xi)}   \right ) ^{\mathrm{T}}  \\
&=   \int_\Theta  {{\phi _i}(\vec \xi){\phi _j}(\vec \xi) \vec A_k^{\mathrm{T}} (\vec u_M(x,t,\vec \xi);\vec \xi)}{\mathrm {d}}  {\mu(\vec \xi)}
=   \int_\Theta  {{\phi _i}(\vec \xi){\phi _j}(\vec \xi) \vec A_k (\vec u_M(x,t,\vec \xi);\vec \xi)}{\mathrm {d}}  {\mu(\vec \xi)}
= \hat{\vec A}_{ij}^{(k)},
\end{align*}
which implies that $\hat{\vec A}_k$  is  symmetric, $k=0,1$.

To show that $\hat{\vec A}_0$ is positive definite, consider
an arbitrary $\vec z:=( \vec z_0, \cdots, \vec z_M  )^{\mathrm{T}} \in
\mathbb{R}^{(M+1)N\times 1}$ with $\vec z_i^{\mathrm{T}} \in
\mathbb{R}^N,i=0,\cdots,M$. Then,
\begin{align*}
\vec z ^{\mathrm{T}} \hat{\vec A}_0 \vec z
& = \sum\limits_{i = 0}^M \sum\limits_{j = 0}^M { \vec z_i }   \hat{\vec A}_{ij}^{(0)}  \vec z_j^{\mathrm{T}   } \\
& = \sum\limits_{i = 0}^M \sum\limits_{j = 0}^M   \vec z_i   \left(  \int_\Theta  {{\phi _i}(\vec \xi){\phi _j}(\vec \xi) \vec A_0 (\vec u_M(x,t,\vec \xi);\vec \xi)}  {\mathrm {d}}  {\mu(\vec \xi)}   \right)  \vec z_j ^{\mathrm{T}}   \\
& = \sum\limits_{i = 0}^M \sum\limits_{j = 0}^M  {    \int_\Theta
\vec z_i
{\phi _i}(\vec \xi){\phi _j}(\vec \xi) \vec A_0 (\vec u_M(x,t,\vec \xi);\vec \xi) \vec z_j ^{\mathrm{T}}    {\mathrm {d}}  {\mu(\vec \xi)}   } \\
& = \sum\limits_{i = 0}^M \sum\limits_{j = 0}^M  {    \int_\Theta
\big( {\phi _i}(\vec \xi) \vec z_i \big)  
 \vec A_0 (\vec u_M(x,t,\vec \xi);\vec \xi) \big ( {\phi _j}(\vec \xi) \vec z_j  \big)^{\mathrm{T}}     {\mathrm {d}}  {\mu(\vec \xi)}   }
 \\
& = \int_\Theta
 \sum\limits_{i = 0}^M \sum\limits_{j = 0}^M  {
\big( {\phi _i}(\vec \xi) \vec z_i \big)  
 \vec A_0 (\vec u_M(x,t,\vec \xi);\vec \xi) \big ( {\phi _j}(\vec \xi) \vec z_j  \big) ^{\mathrm{T}}     }  {\mathrm {d}}  {\mu(\vec \xi)}  \\
 & = \int_\Theta
    {
\left( \sum\limits_{i = 0}^M  {\phi _i}(\vec \xi) \vec z_i \right) 
 \vec A_0 (\vec u_M(x,t,\vec \xi);\vec \xi) \left ( \sum\limits_{j = 0}^M {\phi _j}(\vec \xi) \vec z_j  \right)^{\mathrm{T}}      }  {\mathrm {d}}  {\mu(\vec \xi)}  \ge 0.
\end{align*}
If $\vec z ^{\mathrm{T}} \hat{\vec A}_0 \vec z =0$, then one has
$$
\sum\limits_{i = 0}^M  {\phi _i}(\vec \xi) \vec z_i^{\mathrm{T}} = \vec 0, \quad \mathrm{a.e.}
$$
Since $\{ \phi_i (\vec \xi) \}_{i \in \mathbb{N}\cup\{0\}}$ are basis polynomials, thus $\vec z_i=\vec 0$ for all $i=0,\cdots, M$, i.e. $\vec z=\vec 0$.
The proof is completed. \qed
\end{proof}

An obvious corollary of Theorem \ref{th:symmetric} is that $\hat{\vec B}:=\hat{\vec A}_0^{-1} \hat{\vec A}_1$ is real diagonalizable.

\subsection{Spatial discretization}\label{sec:1D:space}

The fact that the gPC Galerkin system \eqref{eq:1Ddeterministic} is
hyperbolic allows one to use a variety of discretization schemes in
space, e.g., the path-conservation scheme \cite{Castro2006,Pares2006,Castro2009,Castro2013}.
Here, the high-order finite volume WENO scheme given in \cite{Pares2006,Xiong2012} is considered.

For the sake of convenience, the spatial domain $\Omega$ is divided into
the uniform mesh $\left\{ x_{j+\frac{1}{2}}  =  \left( j-\frac{1}{2}\right) \Delta x
\in   \Omega  \left| {j \in \mathbb Z} \right. \right\}$, where
$\Delta x$ denotes the spatial step-size. Multiplying \eqref{eq:1Ddeterministic} by $\hat{\vec A}_0^{-1} $ from the left, and
using the higher-order, path-conservative, finite volume WENO  scheme to \eqref{eq:1Ddeterministic} for spatial discretization may give
\begin{align}\nonumber
 \frac{ {\rm d} \overline{ \hat{\vec U} }_j (t)  } { {\rm d} t }
= & -\frac{1}{\Delta x}  \bigg(
 \hat {\vec B}_{j-\frac{1}{2}}^+ \left(  \hat{\vec U}_{j-\frac{1}{2}}^+(t)  - \hat{\vec U}_{j-\frac{1}{2}}^- (t)  \right)
+ \hat {\vec B}_{j+\frac{1}{2}}^- \left(  \hat{\vec U}_{j+\frac{1}{2}}^+(t)  - \hat{\vec U}_{j+\frac{1}{2}}^- (t)
\right) \bigg)  \\
& - \sum\limits_{m =1}^q \omega_m \hat{\vec B} \big( \hat{\vec U}_j^{\rm WENO} ( x_m ^G,t ) \big) \frac{\partial \hat{\vec U}_j^{\rm WENO} }{\partial x} (x_m ^G,t) =: {\cal L} \left( \overline{ \hat{\vec U} } (t); j \right) ,
\label{eq:1Dspatial}
\end{align}
where $\overline{ \hat{\vec U} }_j (t)$ denotes the cell-averaged approximation of $\hat{\vec U} (x,t)$ over the cell $I_j:=\left[ x_{j-\frac{1}{2}} , x_{j+\frac{1}{2}}  \right]$, $\hat{\vec U}_j^{\rm WENO} ( x,t )$ is a polynomial vector function approximating $\hat{\vec U} ( x,t ) $ in the cell $I_j$ and obtained by using the WENO reconstruction from $\overline{ \hat{\vec U} }_j (t)$, $x_m ^G$ and $\omega_m$  denote the $m$th Gauss-Lobatto node transformed into the interval $I_j$ and the associated weight, respectively, and
$$
\hat{\vec U}_{j-\frac{1}{2}}^+(t) := \hat{\vec U}_j^{\rm WENO} ( x_{j-\frac{1}{2}} + 0,t ),\quad
\hat{\vec U}_{j+\frac{1}{2}}^-(t) := \hat{\vec U}_j^{\rm WENO} ( x_{j+\frac{1}{2}} - 0,t ).
$$
Here the matrix $\hat {\vec B}_{j+\frac{1}{2}}^+$ (resp.
$\hat {\vec B}_{j+\frac{1}{2}}^-$)
has  only non-positive  (resp. non-negative) eigenvalues and
is given by a suitable splitting of the intermediate matrix
\begin{align}
\hat {\vec B}_{\Psi} \left( \hat{\vec U}^-,\hat{\vec U}^+ \right)
&= \left( \sum \limits_{m = 1}^{\tilde q} \tilde \omega_m \hat {\vec A}_0 \big( \vec \Psi \left( s_m ,  \hat{\vec U}^-,\hat{\vec U}^+   \right) \big) \right)^{-1}
\left( \sum \limits_{m = 1}^{\tilde q} \tilde \omega_m \hat {\vec A}_1 \big( \vec \Psi \left( s_m ,  \hat{\vec U}^-,\hat{\vec U}^+   \right) \big) \right),
\label{eq:averageMatrix}\end{align}
with $\vec \Psi \left( s_m ,  \hat{\vec U}^-,\hat{\vec U}^+   \right) :=  \hat{\vec U}^- + s_m \left( \hat{\vec U}^+ -  \hat{\vec U}^- \right)$ and $s_m\in[0,1]$,
that is
$$
 \hat {\vec B}_{\Psi} \left( \hat{\vec U}_{j+\frac{1}{2}}^-,\hat{\vec U}_{j+\frac{1}{2}}^+ \right) = \hat {\vec B}_{j+\frac{1}{2}}^- + \hat {\vec B}_{j+\frac{1}{2}}^+,
$$

Here, the Lax-Friedrichs type splitting is employed, i.e.,
\begin{equation}
\label{eq:LFsplit}
\hat {\vec B}_{j+\frac{1}{2}}^\pm = \frac{1}{2} \Big(  \hat {\vec B}_{\Psi} \left( \hat{\vec U}_{j+\frac{1}{2}}^-,\hat{\vec U}_{j+\frac{1}{2}}^+ \right) \pm \alpha_{j+\frac{1}{2}} {\vec I}_{(M+1)N}  \Big),
\end{equation}
where  ${\vec I}_{(M+1)N}$ denotes the identity matrix of size $(M+1)N$,
and the coefficient
$\alpha_{j+\frac{1}{2}}$ is not less than the spectral radius of the intermediate matrix $\hat {\vec B}_{\Psi}$, which may be estimated by using
the eigenvalues of the Jacobian matrix
$\vec A(\vec U; \vec \xi)$ of \eqref{eq:1D},
see  Theorem \ref{Thm2.2} given below.

Practically, the intermediate matrix $\hat {\vec B}_{\Psi} \left( \hat{\vec U}^-,\hat{\vec U}^+ \right)$ is an approximation of the matrix
\begin{align} \nonumber
\left( \int_0^1{ \hat {\vec A}_0 \big( \vec \Psi \left( s ,  \hat{\vec U}^-,\hat{\vec U}^+   \right) \big) } {\rm d} s \right)^{-1}
\left( \int_0^1{ \hat {\vec A}_1 \big( \vec \Psi \left( s ,  \hat{\vec U}^-,\hat{\vec U}^+   \right) \big) } {\rm d} s \right),
\end{align}
by using the Gaussian quadrature along
 the integral path $\vec \Psi \left( s ,  \hat{\vec U}^-,\hat{\vec U}^+   \right) =  \hat{\vec U}^- + s \left( \hat{\vec U}^+ -  \hat{\vec U}^- \right)$,
 $s \in [0,1]$, which is the segment between $\hat{\vec U}^-$ and $\hat{\vec U}^+$,
and  $s_m $ and $\tilde \omega_m $  denote the $m$th Gaussian node transformed into the interval $[0,1]$ and the associated weight, respectively.
It is not difficult to prove by using Theorem \ref{th:symmetric} that the intermediate matrix $\hat {\vec B}_{\Psi} \left( \hat{\vec U}^-,\hat{\vec U}^+ \right)$ is real diagonalizable if $\vec \Psi \left( s_m ,  \hat{\vec U}^-,\hat{\vec U}^+   \right),~m=1,\cdots,\tilde q$, are admissible, i.e.
\begin{equation}
\vec u_M^{(m)}(\vec \xi) = \sum\limits_{i = 0}^M
\big({{\vec {\hat u}_i^{(m)}}{\phi _i}(\vec \xi)} \big)^{\mathrm{T}}\in {\cal G}, \quad \forall \vec \xi \in \Theta,
\label{eq-uMm}
\end{equation}
where $ \left( \vec {\hat u}_0^{(m)}, \cdots,  \vec {\hat u}_M^{(m)}  \right)^{\mathrm{T}} :=
\vec \Psi \left( s_m ,  \hat{\vec U}^-,\hat{\vec U}^+   \right)$.



Before estimating the upper bound of the spectral radius of $\hat {\vec B}_{\Psi}$, 
we first  prove a lemma.

\begin{lemma}\label{lam:A0A1}
If  $\vec {\cal A}_0$ is a real symmetric and positive-definite matrix, and $\vec {\cal A}_1$ is a real symmetric matrix, then $\lambda \vec {\cal A}_0 \pm \vec {\cal A}_1 $ is positive semi-definite if and only if
 $\lambda \ge \varrho\left( \vec {\cal A}_0^{-1} \vec {\cal A}_1 \right)$, where $\varrho\left( \vec {\cal A}_0^{-1} \vec {\cal A}_1 \right)$ denotes the spectral radius of the matrix  $\vec {\cal A}_0^{-1} \vec {\cal A}_1$, and $\lambda$ is a real number.
\end{lemma}

\begin{proof}
Because
$
\lambda \vec {\cal A}_0 \pm \vec {\cal A}_1 = \left( \vec {\cal A}_0^{\frac{1}{2}} \right)^T \left( \lambda \vec I \pm \vec {\cal A}_0^{\frac{1}{2}}
\vec {\cal A}_0^{-1} \vec {\cal A}_1 \vec {\cal A}_0^{-\frac{1}{2}} \right) \vec {\cal A}_0^{\frac{1}{2}},
$
where $\vec I$ is the identity matrix,
 $\lambda \vec {\cal A}_0 \pm \vec {\cal A}_1$ is  congruent to $\lambda \vec I \pm \vec {\cal A}_0^{\frac{1}{2}}
\vec {\cal A}_0^{-1} \vec {\cal A}_1 \vec {\cal A}_0^{-\frac{1}{2}}$.
The hypothesis and Theorem \ref{th:symmetric} imply that
$\vec {\cal A}_0^{-1} \vec {\cal A}_1$ is real diagonalizable.
Thus $\lambda \vec {\cal A}_0 \pm \vec {\cal A}_1 $ is positive semi-definite if and only if $\lambda \vec I \pm \vec {\cal A}_0^{\frac{1}{2}}
\vec {\cal A}_0^{-1} \vec {\cal A}_1 \vec {\cal A}_0^{-\frac{1}{2}}$ is positive semi-definite, equivalently, $\lambda \ge \varrho \left( \vec {\cal A}_0^{\frac{1}{2}}
\vec {\cal A}_0^{-1} \vec {\cal A}_1 \vec {\cal A}_0^{-\frac{1}{2}}\right) =  \varrho \left(
\vec {\cal A}_0^{-1} \vec {\cal A}_1 \right) $. The proof is completed. \qed
\end{proof}

\begin{theorem}\label{Thm2.2}
If 
$$
\vec \Psi \left( s_m ,  \hat{\vec U}^-,\hat{\vec U}^+   \right) \in
{\hat{\cal G}_M:=}
 \left\{ \hat{\vec U} := ( \vec {\hat u}_0, \cdots,  \vec {\hat u}_M  )^{\mathrm{T}} \in \mathbb{R}^{(M+1)N} \left| ~ \sum\limits_{i = 0}^M
{ \big({{\vec {\hat u}_i }{\phi _i}(\vec \xi)} \big)^{\mathrm{T}} }
  \in {\cal G},~\forall \vec \xi \in \Theta \right. \right\},
$$
for $m=1,\cdots,\tilde q$,
  then
the spectral radius of the intermediate matrix $\hat {\vec B}_{\Psi}$ satisfies
\begin{equation}\label{eq:upperbound-2}
\alpha  :=  \max\limits_{\ell,m} \sup \limits_{\vec \xi \in \Theta} \big\{ \left| \lambda_\ell \left(\vec u_M^{(m)}(\vec \xi);\vec \xi  \right) \right| \big\}
 \ge \varrho\left( \hat {\vec B}_{\Psi} \right),
\end{equation}
where
 $\lambda_\ell$ is the $\ell$--th eigenvalue of the Jacobian matrix
$\vec A(\vec U; \vec \xi)$ of \eqref{eq:1D}, $\ell=1,2,\cdots,N$.
\end{theorem}

\begin{proof}
Under the hypothesis,
it is easy to show that $\vec u_M^{(m)}(\vec \xi)$ defined in \eqref{eq-uMm} belongs to ${\cal G}$, for all $\vec \xi \in \Theta$. Thanks to Theorem \ref{th:symmetric},  $\hat {\vec A}_0 \left( \vec \Psi \left( s_m ,  \hat{\vec U}^-,\hat{\vec U}^+   \right) \right)$ is real symmetric and positive definite, and $\hat {\vec A}_1 \left( \vec \Psi \left( s_m ,  \hat{\vec U}^-,\hat{\vec U}^+   \right) \right)$ is real symmetric. Because
$$
\hat {\vec A}_{k}^\Psi := \sum \limits_{m = 1}^{\tilde q} \tilde \omega_m \hat {\vec A}_k \big( \vec \Psi \left( s_m ,  \hat{\vec U}^-,\hat{\vec U}^+   \right) \big),
\
k=0,1,
$$
are two convex combinations of $\hat {\vec A}_k \left( \vec \Psi \left( s_m ,  \hat{\vec U}^-,\hat{\vec U}^+   \right) \right),~m=1,\cdots,\tilde q$,
 $\hat {\vec A}_{0}^{\Psi}$ is real symmetric and positive definite, and $\hat {\vec A}_{1}^{\Psi}$ is real symmetric.
Due to Lemma \ref{lam:A0A1}, the inequality \eqref{eq:upperbound-2}
is equivalent to  the positive  semi-definiteness of the matrix $\alpha \hat {\vec A}_{0}^{\Psi} \pm  \hat {\vec A}_{1}^{\Psi}$, which is a block matrix of the form
$$
\alpha \hat {\vec A}_{0}^{\Psi} \pm  \hat {\vec A}_{1}^{\Psi} =\left( {\begin{array}{*{20}{c}}
   {\vec A}_{00}^{\Psi} &  \cdots  & {\vec A}_{0M}^{\Psi}  \\
    \vdots  & {} &  \vdots   \\
   {\vec A}_{M0}^{\Psi} &  \cdots  & {\vec A}_{MM}^{\Psi}
\end{array}} \right),
$$
where the blocks
$$
{\vec A}_{ij}^{\Psi} = \sum \limits_{m = 1}^{\tilde q} \tilde \omega_m \int_\Theta  {{\phi _i}(\vec \xi){\phi _j}(\vec \xi)
\left(
\alpha
{\vec A_0}(\vec u_M^{(m)}(\vec \xi);\vec \xi)
\pm {\vec A_1}(\vec u_M^{(m)}(\vec \xi);\vec \xi)
\right)
}
{\mathrm {d}} {\mu(\vec \xi)} \in {{\mathbb R}^{N \times N}}.
$$
The matrix $\alpha
{\vec A_0}(\vec u_M^{(m)}(\vec \xi);\vec \xi)
\pm {\vec A_1}(\vec u_M^{(m)}(\vec \xi);\vec \xi)$ is positive semi-definite for all $\vec \xi \in \Theta$ due to the definition of $\alpha$
and Lemma \ref{lam:A0A1}.
 Then,  the positive semi-definiteness of $\alpha \hat {\vec A}_{0}^{\Psi} \pm  \hat {\vec A}_{1}^{\Psi}$ may be concluded by following the proof of Theorem \ref{th:symmetric}. The proof is completed. \qed
\end{proof}

Theorem \ref{Thm2.2} indicates that the complicate  calculation of the spectral radius of the large matrix $\hat {\vec B}_{\Psi}$ may be avoided in the Lax-Friedrichs type splitting \eqref{eq:LFsplit} and the
CFL condition for determining the time step-size in practical computations.

\begin{remark} \label{rem:limiter}
In our computations, the parameters $q$ in \eqref{eq:1Dspatial} and $\tilde q$ in \eqref{eq:averageMatrix} is taken as $4$ and 3, respectively.
 The function $\hat{\vec U}_j^{\rm WENO} ( x,t )$ is
derived by using the Lagrange interpolation based on $\left\{ \hat{\vec U}_j^{\rm WENO} ( x_m ^G,t ) \right\}_{m=1}^q$, which
are obtained by the fifth-order accurate WENO reconstruction from $\{\overline{ \hat{\vec U} }_j (t)\}$.
If $\overline{ \hat{\vec U} }_j (t) \in \hat {\cal G}_M$  but $\hat{\vec U}_j^{\rm WENO} ( x_m ^G,t )\notin \hat {\cal G}_M$ for at least one
Gauss-Lobatto point $x_m ^G$, then
$\hat{\vec U}_j^{\rm WENO} ( x_m ^G,t )$ is limited
by the limiting procedure
$$
\widetilde {\vec U}_j^{\rm WENO} ( x_m ^G,t ) = \theta \left(  \hat {\vec U}_j^{\rm WENO} ( x_m ^G,t ) - \overline{ \hat{\vec U} }_j (t) \right) + \overline{ \hat{\vec U} }_j (t), \quad m=1,\cdots,q,
$$
where $\theta = \min \limits_{1\le m \le q} \{\theta_m\}$, and  $\theta_m= 1$ for $\hat{\vec U}_j^{\rm WENO} ( x_m ^G,t ) {\in} \hat {\cal G}_M$; otherwise, $\theta_m$ corresponds to the intersection point between the line segment $\{\vec \Psi \left( s ,  \overline{ \hat{\vec U} }_j (t) ,  \hat {\vec U}_j^{\rm WENO} ( x_m ^G,t )   \right), s\in [0,1]\}$ and the boundary of $\hat{\cal G}_M$.
If the admissible state set $\cal G$ is convex, e.g. for the Euler equations \cite{zhang2010} or the relativistic hydrodynamical equations \cite{WuTang2015},
then $\hat{\cal G}_M$ is also convex, and thus  the intersection point is unique.
However, for the system \eqref{eq:1D}, the above gPC-SG method cannot generally preserve
the property that $\overline{ \hat{\vec U} }_j (t) \in \hat {\cal G}_M$.
Example \ref{example:Sod} in Section \ref{sec:experiments} will show
 that the solutions $\overline{ \hat{\vec U} }_j (t) \notin \hat {\cal G}_M$
  may appear for very few cells in few time steps.
For this case,  $\overline{ \hat{\vec U} }_j (t)=:( \overline{ \vec {\hat u}}_0, \cdots,  \overline {\vec {\hat u}}_M  )^{\mathrm{T}}$ may be further limited as $\widetilde{ \hat{\vec U} }_j (t;\tilde\theta)=( \overline{ \vec {\hat u}}_0, \tilde\theta \overline {\vec {\hat u}}_1 ,\cdots, \tilde\theta \overline {\vec {\hat u}}_M  )^{\mathrm{T}},$
where $\tilde \theta$ corresponds to the intersection point between
 the line segment $\{\vec \Psi \left( s , \widetilde{ \hat{\vec U} }_j (t;0) , \overline{ \hat{\vec U} }_j (t)   \right), s \in [0,1] \}$ and the boundary of $\hat{\cal G}_M$.
\end{remark}

\subsection{Time discretization}\label{sec:1D:time}

The time derivatives in the semi-discrete system \eqref{eq:1Dspatial}
 can be approximated  any proper method, e.g., the explicit total
 variation diminishing Runge-Kutta method \cite{Shu1988}.
Here, the third-order
accurate version is considered, as an example. It takes the following form
\begin{align} \label{eq:RK1} \begin{aligned}
& \hat {\vec U}^ *_j   = \hat {\vec U}^n_j  + \Delta t_n {\cal L}( \hat {\vec U}^n;j ), \\
& \hat {\vec U}^{ *  * }_j  = \frac{3}{4} \hat {\vec U}^n_j  + \frac{1}{4}\Big( \hat {\vec U}^ *_j
 + \Delta t_n {\cal L}(\hat {\vec U}^ *;j  )\Big), \\
& \hat {\vec U}^{n+1}_j  = \frac{1}{3} \hat {\vec U}^n_j  + \frac{2}{3}\Big( \hat {\vec U}^{ *  * }_j
+ \Delta t_n {\cal L}( \hat {\vec U}^{ *  * };j)\Big),
\end{aligned}\end{align}
where $ \Delta t_n$ denotes the time step-size.


\section{Two-dimensional gPC-SG method}\label{sec:2D}

This section extends the above gPC-SG method to a general
 two-dimensional (2D) quasilinear hyperbolic system
\begin{equation}\label{eq:2D}
\frac{ \partial } {\partial t}  \vec U(x,y,t,\vec \xi)  +  \frac{ \partial } {\partial x} \vec F(\vec U(x,y,t,\vec \xi);\vec \xi )
+ \frac{ \partial } {\partial y} \vec G(\vec U(x,y,t,\vec \xi);\vec \xi )
= \vec 0, \quad  (x,y) \in \Omega,~  t >0,
\end{equation}
If the eigenvector matrices of the Jacobian matrices in $x$- and $y$-directions are different,
the general  system \eqref{eq:2D} can not be symmetrized
  with the approach used in  Section \ref{sec:1D:symm}.
To derive an efficient gPC-SG method for the general system \eqref{eq:2D},
 the operator splitting technique is employed here,
which in essence reduces the multi-dimensional problem
to a sequence of augmented 1D problems. For example,
the 2D system \eqref{eq:2D}
is decomposed into two subsystems (the $x$- and $y$-split 2D systems)
as follows
\begin{equation}\label{eq:2Dx}
\frac{ \partial } {\partial t}  \vec U(x,y,t,\vec \xi)  +  \frac{ \partial } {\partial x} \vec F(\vec U(x,y,t,\vec \xi);\vec \xi )
= \vec 0,
\end{equation}
and
\begin{equation}\label{eq:2Dy}
\frac{ \partial } {\partial t}  \vec U(x,y,t,\vec \xi)
+ \frac{ \partial } {\partial y} \vec G(\vec U(x,y,t,\vec \xi);\vec \xi )
= \vec 0.
\end{equation}
For the splited sub-systems, the approach presented in Section \ref{sec:1D} can be used to obtain a deterministic symmetrically hyperbolic system for the gPC expansion coefficients.

 Assuming that the 2D spatial domain $\Omega$ is
 divided into the uniform rectangular mesh
 $\left\{ \left(x_{j+\frac{1}{2}}  =  \left(j+\frac{1}{2}\right) \Delta x, y_{k+\frac{1}{2}}= \left( k+\frac{1}{2} \right) \Delta y \right)
\in   \Omega  \left| {j,k \in \mathbb Z} \right. \right\}$,
and  the cell-averaged approximation of the ``initial'' gPC expansion coefficients $\overline{ \hat{\vec U} }_{j k}^n$
over the cell $\left[ x_{j-\frac{1}{2}} , x_{j+\frac{1}{2}}  \right]\times \left[ y_{k-\frac{1}{2}} , y_{k+\frac{1}{2}}  \right] $
at $t=t_n$ are given,
then the cell-averaged approximation of the expansion coefficients at $t=t_n+\Delta t_n$ can approximately calculated based on
some higher-order accurate operator splitting method, e.g.
a third-order accurate operator splitting method \cite{Sornborger1999,Thalhammer2009}
\begin{equation}\label{eq:split2Dx}
\overline{ \hat{\vec U} }_{j k}^{n+1}  = {\cal E}_x^{\tau _1 }
{\cal E}_y^{\tau _1  + \tau _2 }
{\cal E}_x^{\tau _2 }
{\cal E}_y^{\tau _3 }
{\cal E}_x^{\tau _3  + \tau _4 }
{\cal E}_y^{\tau _4 } \overline{ \hat{\vec U} }_{j k}^n,
\end{equation}
or
\begin{equation}\label{eq:split2Dy}
\overline{ \hat{\vec U} }_{j k}^{n+1}  = {\cal E}_y^{\tau _1 }
{\cal E}_x^{\tau _1  + \tau _2 }
{\cal E}_y^{\tau _2 }
{\cal E}_x^{\tau _3 }
{\cal E}_y^{\tau _3  + \tau _4 }
{\cal E}_x^{\tau _4 } \overline{ \hat{\vec U} }_{j k}^n ,
\end{equation}
where ${\cal E}_x^\tau$ and ${\cal E}_y^\tau$ denote the 1D finite volume WENO scheme for the deterministic systems corresponding to the split systems
 \eqref{eq:2Dx}
and \eqref{eq:2Dy}, respectively, and the ``time step-sizes'' $\tau_i$ are
\begin{align*}
   & \tau _1  = \frac{{2\Delta t_n}}{{5 - \sqrt {13}  + \sqrt {2(1 + \sqrt {13} )} }}, \ \tau _2  = \frac{{7 + \sqrt {13}  - \sqrt {2(1 + \sqrt {13} )} }}{{12}}\Delta t_n, \\
   & \tau _3  = \frac{{\tau _1^2 }}{{\tau _2  - \tau _1 }}, \qquad \tau _4  = \Delta t_n - \left( {\tau _1  + \tau _2  + \tau _3 } \right).
\end{align*}



\section{Numerical experiments}
\label{sec:experiments}

This section presents several numerical examples to verify the accuracy and effectiveness of the proposed gPC-SG methods.
The quasilinear system of hyperbolic conservation laws are taken as the 1D  Euler equations
\begin{equation}\label{eq:1DEuler}
\frac{\partial \vec U}{\partial t} + \frac{\partial \vec F ( \vec U )}{\partial x} = 0,
\end{equation}
with
\[
\vec U = (\rho ,\rho u,  E)^{\rm T} , \ \ \vec F  = (\rho u,\rho u^2  + p,  uE + up)^{\rm T},
\]
and the 2D  Euler equations
\begin{equation}\label{eq:2DEuler}
\frac{\partial \vec U}{\partial t} + \frac{\partial \vec F_1 ( \vec U )}{\partial x} +  \frac{\partial \vec F_2 ( \vec U )}{\partial y}= 0,
\end{equation}
with
\[
\vec U = (\rho ,\rho u, \rho v, E)^{\rm T} , \  \vec F_1  = (\rho u,\rho u^2  + p,\rho u v, uE + up)^{\rm T},\
\vec F_2  = (\rho v,\rho u v, \rho v^2  + p, vE + vp)^{\rm T}.
\]
Here $\rho,u,v, p$, and $E$ denote the density,  the velocity
components in $x$- and $y$-directions, the pressure, and the total
energy, respectively. In 1D, $E=\rho e + \frac{1}{2}\rho u^2 $; and in 2D,
$E=\rho e + \frac{1}{2} \rho (u^2 + v^2)$, where $e$ is the
internal energy.
An equation of state is needed to close the system \eqref{eq:1DEuler}
or \eqref{eq:2DEuler}.
We focus on the case for the ideal gases, i.e.
\begin{equation}
  \label{eq:GammaLaw}
  p = (\Gamma-1)\rho e,
\end{equation}
where $\Gamma$ is the adiabatic index. In the numerical experiments, assume  that the uncertainty may
enter the problems through the initial or boundary conditions or the adiabatic index in the equation of state, more specifically,
the adiabatic index or the initial $\rho,u,v$ or $p$ or boundary condition depends on a 1D random
variable $\xi$, which is assumed to obey the uniform distribution on $[-1,1]$ for simplicity.
The Legendre polynomials are taken as the gPC basis, and thus the
mean and variance of the gPC solution $\vec u_M$ in \eqref{eq:gPCAPP} are respectively given by
$$
{\mathbb{E}} [\vec u_M] = \hat{\vec u}_0,\quad {\rm Var} [\vec u_M] = \sum \limits_{i=1}^M  \hat{\vec u}_M^2 ,
$$
and the corresponding standard deviation $\sigma [\vec u_M] = \sqrt{{\rm Var} [\vec u_M]}$.
Unless specifically stated, all computations will use the CFL number
of 0.6 in the finite volume WENO schemes for the deterministic gPC-SG
systems. Due to the complex nature of the solution in physical space,
we focus exclusively on the $d=1$ case in random space. This allows us
to fully resolve the solution in random space. Extension to multiple
random variables is straightforward and poses no numerical
difficulty. It merely increases the simulation time significantly.

\subsection{1D case}

\begin{example}[Smooth problem]\label{example:accuracy}\rm
This example is used to check the accuracy of the gPC-SG method for
smooth solution
  \begin{equation*}
    (\rho,u,p) (x,t,\xi) = \Big(1+0.2\sin\big(2\pi(x-(0.8+0.2 \xi)t)\big),0.8+0.2 \xi,1   \Big),
  \end{equation*}
with randomness, which describes a sine wave propagating periodically within
  the spatial domain $[0,1]$ with uncertain velocity.

  The spatial domain is divided into $N_c$ uniform cells and the
  periodic boundary conditions are specified.
  The time step is taken as $\Delta t_n = \Delta x^{\frac{5}{3}}$ in order to realize fifth-order accuracy in time in the present case.
  The $l^1$-errors in the
mean and standard deviation of the density at $t=0.2$ obtained by the gPC-SG method with $N_c=320$ uniform cells and different gPC orders $M$ are plotted in Fig. \ref{fig:gPCM}.
  The fast exponential convergence with respect to the order of gPC expansion is observed both in mean and standard deviation.
  The errors saturate at modest gPC orders, because the spatial and time discretization errors become dominant at this stage.
  Table \ref{tab:order} lists the $l^1$-errors at $t=0.2$ in the
mean and standard deviation of the density and corresponding convergence rates for
  the gPC-SG method with $M=4$ and different $N_c$.
  The results show that the convergence rate of fifth-order
   can be almost obtained in space and time.
\end{example}

\begin{figure}[htbp]
  \centering
  \includegraphics[height=0.48\textwidth]{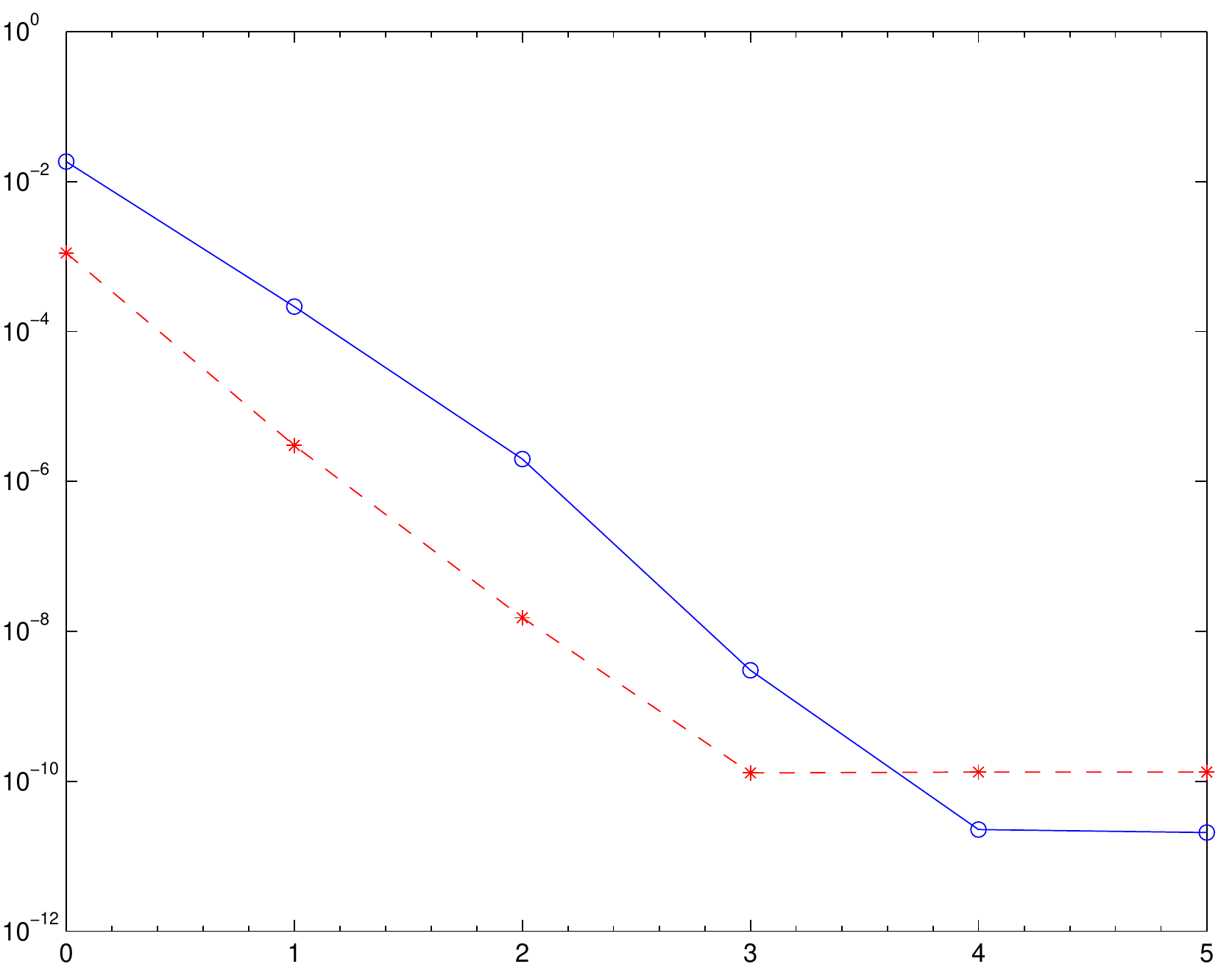}
  \caption{\small Example \ref{example:accuracy}: $l^1$-errors at $t=0.2$ in mean (dashed lines with symbols ``${\color{red}*}$'') and standard deviation (solid lines with symbols ``{\color{blue}$\circ$}'') of the density, with respect to gPC order $M$ for the gPC-SG method with $320$ unform cells.
 }
  \label{fig:gPCM}
\end{figure}

\begin{table}[htbp]
  \centering
  \caption{\small Example \ref{example:accuracy}:
   $l^1$-errors at $t=0.2$  in the
mean and standard deviation of the density and corresponding convergence rates for
  the gPC-SG method with $M=4$.
  }
  \label{tab:order}
\begin{tabular}{|c||c|c||c|c|}
  \hline
\multirow{2}{3pt}{$N_c$}
 &\multicolumn{2}{c||}{Mean of $\rho$}&\multicolumn{2}{c|}{Standard deviation of $\rho$}\\
 \cline{2-5}
 &$l^1$ error&$l^1$ order &$l^1$ error&$l^1$ order\\
 \hline
10& 3.1144e-3& --          &4.4610e-4& --   \\
20& 1.4266e-4& 4.4483    & 2.1666e-5&  4.3639 \\
40& 4.3836e-6& 5.0243     & 9.4766e-7& 4.5149 \\
80& 1.3642e-7 & 5.0060    & 2.8874e-8& 5.0366 \\
160& 4.2527e-9 & 5.0035  & 7.6170e-10& 5.2444\\
320&  1.3279e-10 & 5.0012  &2.2683e-11& 5.0695 \\
\hline
  \end{tabular}
  \end{table}

\begin{example}[Uncertain boundary condition problem]\label{example:BC}\rm
Initially, the spatial domain $[0,1]$ is filled with static fluid with unit density, pressure of 0.6 and adiabatic index of $\frac{5}{3}$.
Waves are excited by a time periodic driver which acts at the left boundary $x=0$, i.e.,
\begin{equation}\label{eq:1DleftBC}
(\rho,u,p)(0,t) = (1, 0.02 \sin( 2\pi w t ),0.6),
\end{equation}
while outflow boundary condition is specified at the right $x=5$. Similar problems are considered in \cite{fuchs2010,Kappeli2014}.
Here, the parameter $w$ describing the frequency of the waves
contains uncertainty as follows
$w(\xi) = 1+0.1 \xi$.
Fig.~\ref{fig:UBC} gives mean and standard deviation of the density at time $t=4$ by using the gPC-SG method with $M=3$ and 200 uniform cells,
where the solid lines represent the reference solutions given by a collocation method with 1000 uniform cells. Here, the collocation method takes 40 Gaussian points as collocation points, and solves the deterministic problem for each collocation point by using the fifth-order accurate finite difference WENO scheme and the third-order Runge Kutta method \eqref{eq:RK1} for time discretization.
Good agreements between the numerical solutions and the reference solutions can be observed in these results, but we can observe that the uncertainty in frequency influences the local peak and valley values of the density.
It is different from the results in Fig. \ref{fig:UBC_fixw}, where
the same peak and valley values
are observed in the deterministic problems with different $w$.

\begin{figure}[htbp]
  \centering
  \includegraphics[height=0.38\textwidth]{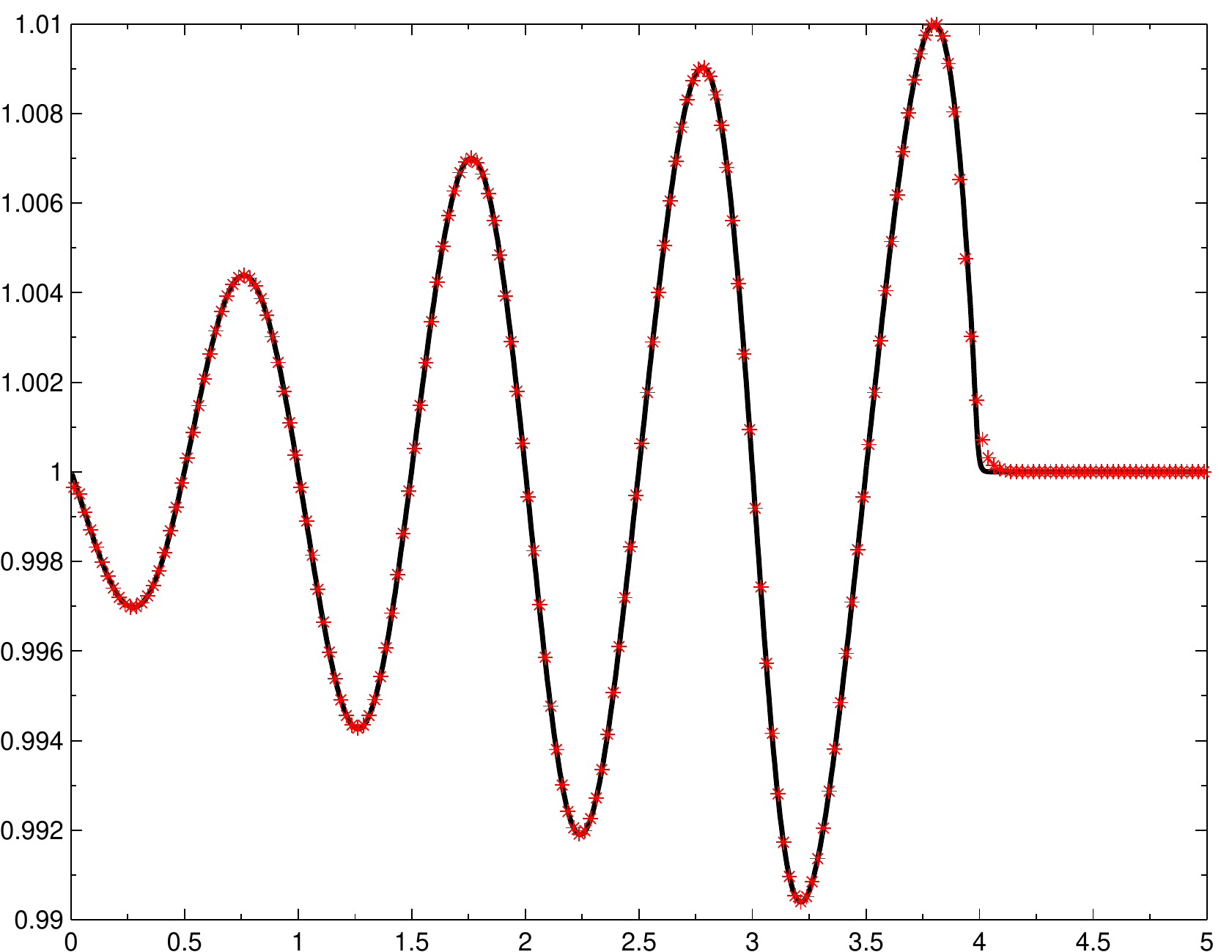}
  \includegraphics[height=0.39\textwidth]{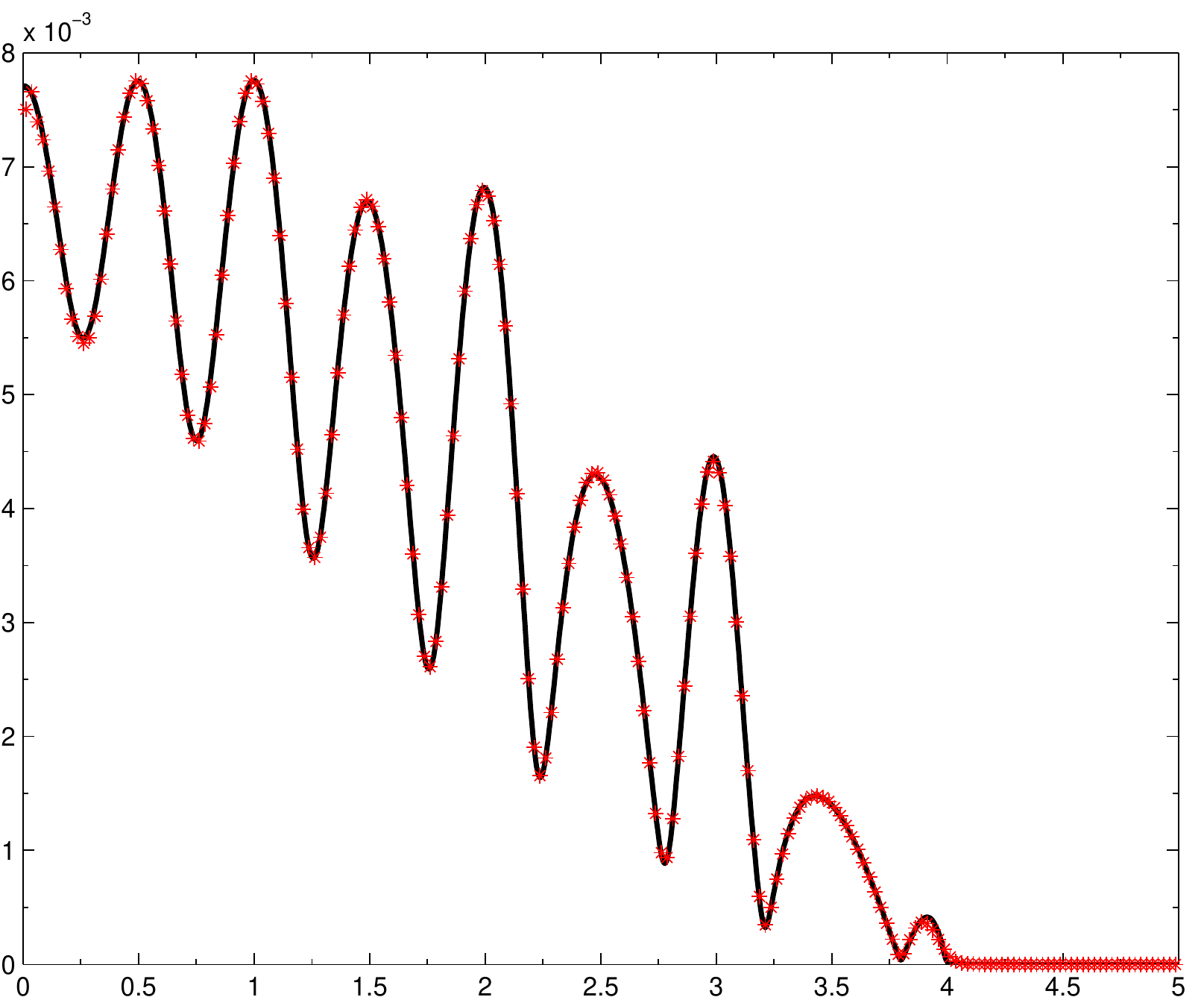}
  \caption{\small Example \ref{example:BC}: the mean and standard deviation of $\rho$ at $t=4$ by using the gPC-SG method (``$\color{red}*$'') with $M=3$ and 200 uniform cells, and the solid lines represent the reference solutions given by a collocation method with 1000 uniform cells.
 }
  \label{fig:UBC}
\end{figure}

\begin{figure}[htbp]
  \centering
  \includegraphics[height=0.38\textwidth]{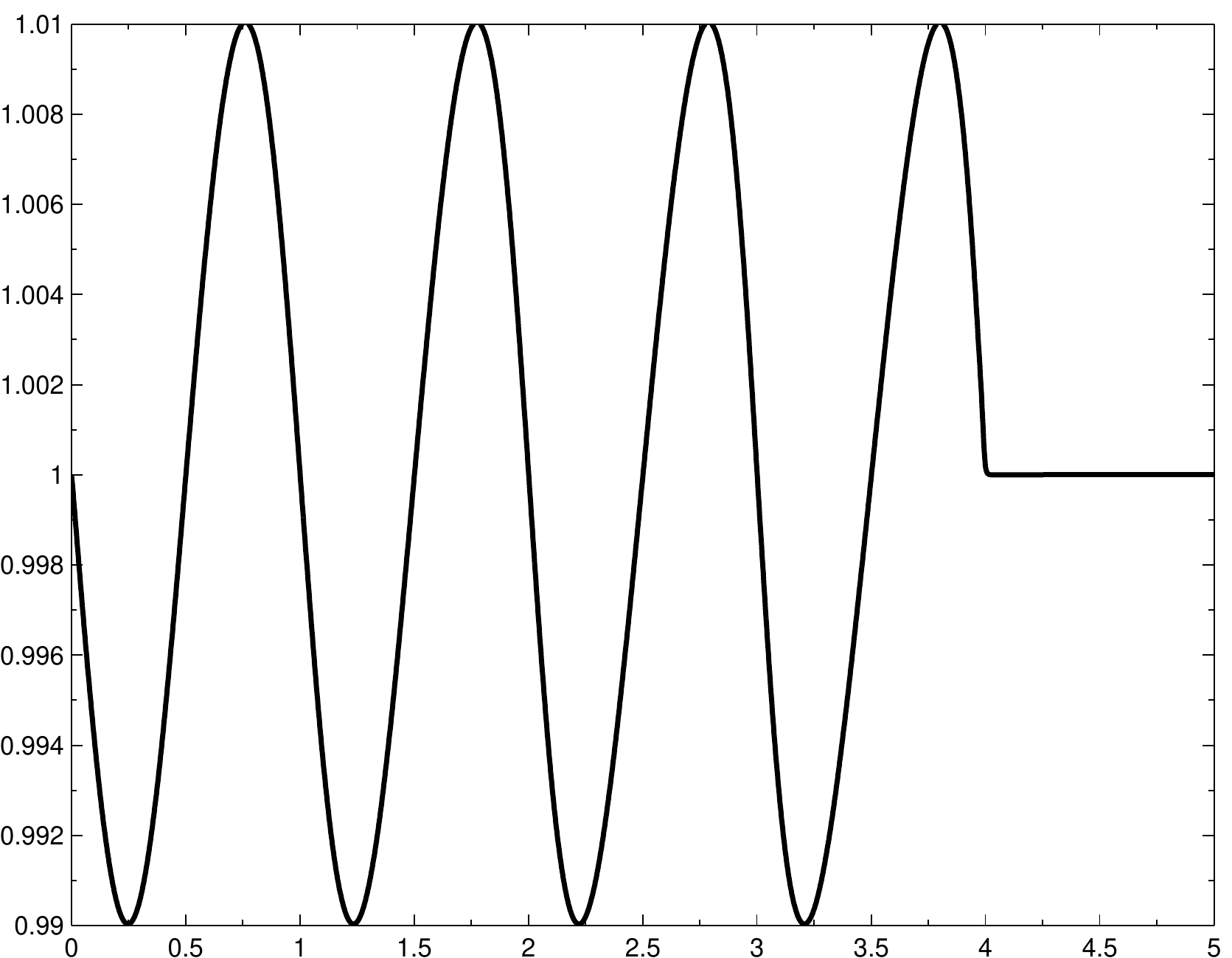}
  \includegraphics[height=0.38\textwidth]{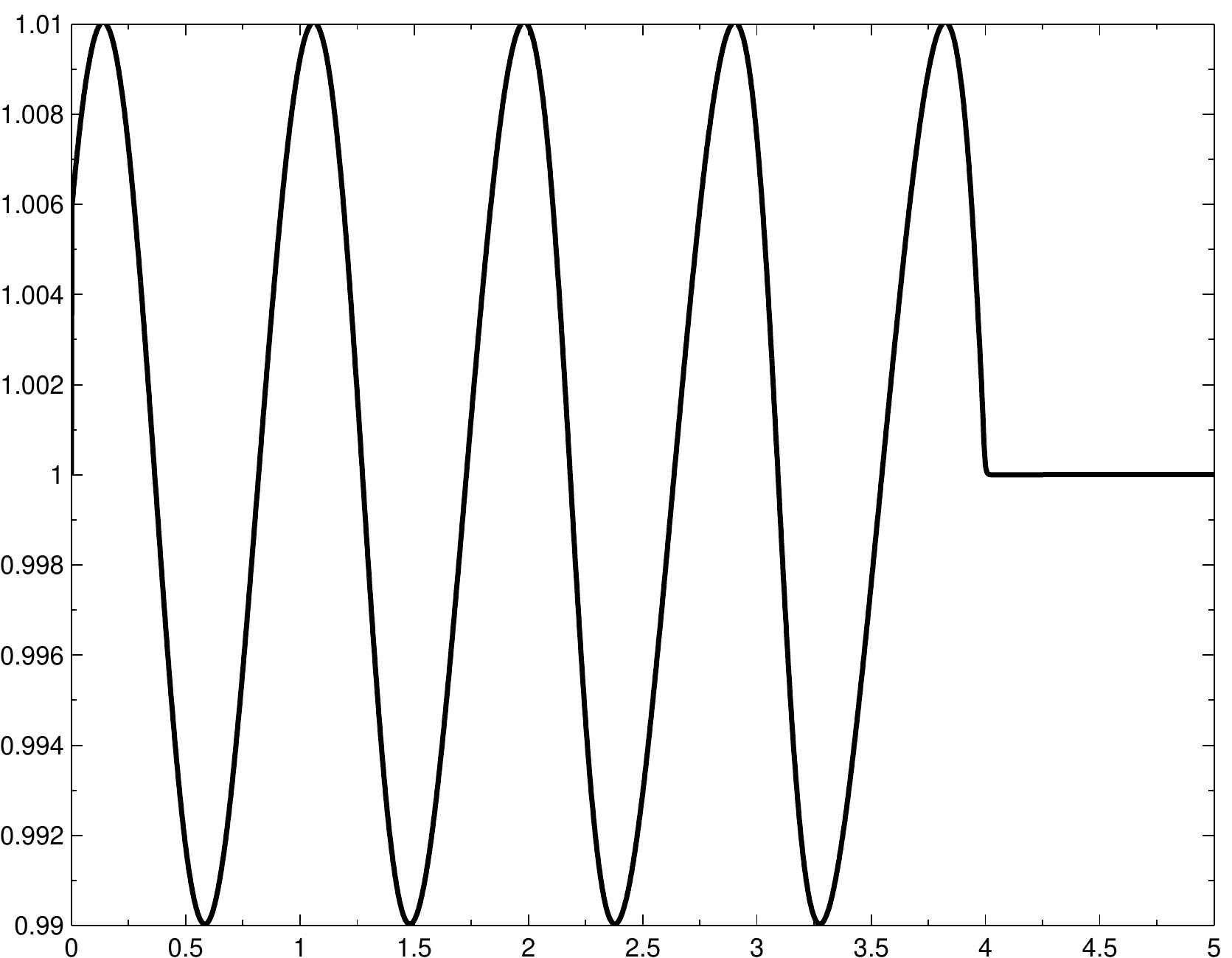}
  \caption{\small Example \ref{example:BC}: $\rho$ at $t=4$ for deterministic $w$ cases: $w=1$ (left), $w=1.1$ (right).
 }
  \label{fig:UBC_fixw}
\end{figure}

\end{example}

\begin{example}[Sod problem with uncertain initial condition]\label{example:Sod}\rm
This test considers the Sod shock tube problem with uncertainty in the location of the initial discontinuity. The initial data are
  \begin{equation}
    \label{eq:Sod2}
    (\rho,u,p)(x,0,\xi)=
    \begin{cases}
      (1,0,1),\ \ & x < 0.5+0.05 \xi,\\
      (0.125,0,0.1),\ \ & x>0.5+0.05 \xi,
    \end{cases}
  \end{equation}
and $\Gamma = 1.4$. The same setup is considered in \cite{Poette2009,Despres2013,Pettersson2014,Chertock:Euler}.
The gPC-SG method may easily fail in this test due to the appearance of negative density caused by the oscillations \cite{Poette2009}.
In our computations, numerical solutions with negative density and pressure in few cells are encountered in the first few steps, the positivity-preserving technique presented in Remark \ref{rem:limiter} is used (only in this test) to deal with such difficulty.
Fig. \ref{fig:SodIN} displays numerical means and standard deviations of the density
  at $t=0.18$, which are obtained by using the gPC-SG method with 200
  uniform cells and $M=8$, where the solid lines represent the reference solutions given by the exact Riemann solver with 64 Gaussian points in the random space to evaluate the mean and standard deviation.
Three sharp fronts with large standard deviation are observed, and their locations are corresponding to the rarefaction fan, contact discontinuity and shock from left to right. One can see small oscillations in the numerical standard deviation, which is also observed in \cite{Despres2013}. This Gibbs phenomenon results from the discontinuity of the solution in random space, and can be improved by utilizing piecewise approximations in random space, such as multi-element gPC \cite{Wan2005} or wavelet basis \cite{Maitre2004,Pettersson2014}.

\begin{figure}[htbp]
  \centering
  \includegraphics[height=0.38\textwidth]{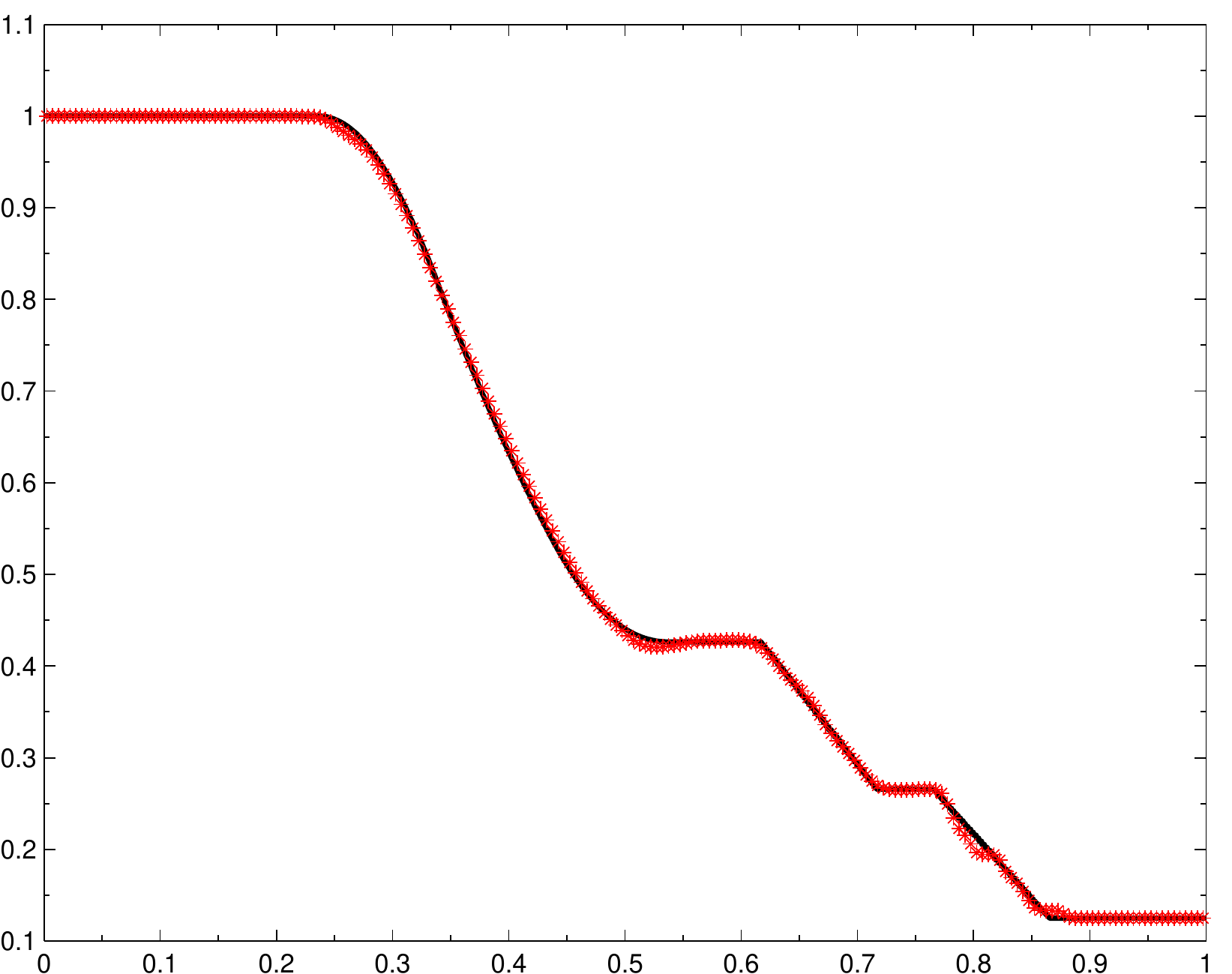}
  \includegraphics[height=0.38\textwidth]{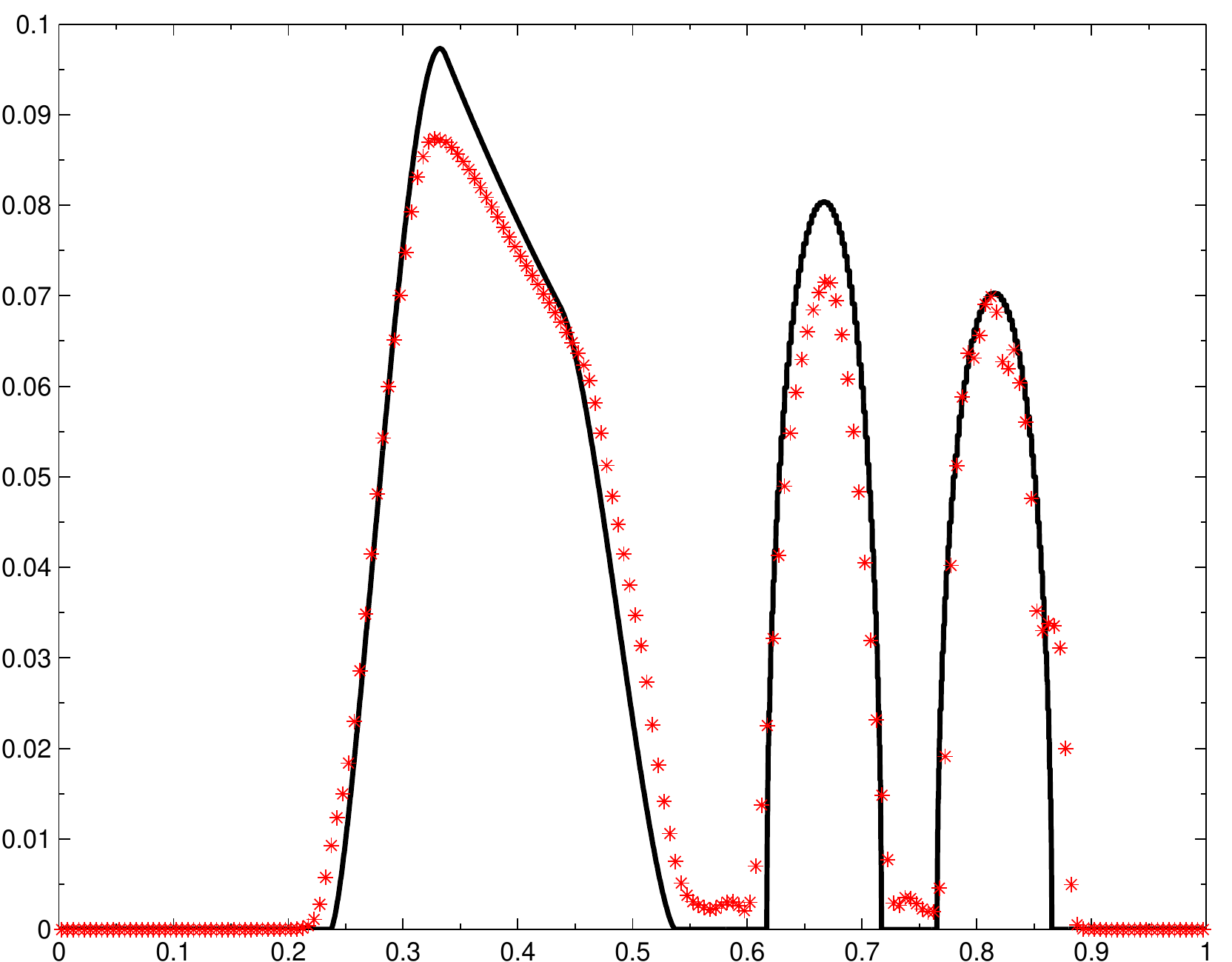}
  \caption{\small Example \ref{example:Sod}: the mean and standard deviation of the density at time $t=0.18$ by using the gPC-SG method (``$\color{red}*$'') with $M=8$ and 200 uniform cells, and the solid lines represent the reference solutions.
 }
  \label{fig:SodIN}
\end{figure}

\end{example}

\subsection{2D case}

The 2D Riemann problems are theoretically studied
for the first time in \cite{zhang1990}.
Since then, they  become the benchmark tests for verifying the
accuracy and resolution of numerical schemes, see
\cite{SchulzRinne1993,LaxLiu1998,HanLiTang2011,WuYangTang2014b,WuTang2014}.

Four 2D Riemann problems in spatial domain $[0,1]\times[0,1]$ with uncertain initial data or adiabatic index of 2D Euler equations  \eqref{eq:2DEuler} are considered here.


\begin{example}[2D Riemann problems I and II]\label{example:RP1}\rm
The first two problems are two perturbed versions of a classic deterministic Riemann problem with initial conditions \cite{LaxLiu1998}
\begin{equation}\label{eq:2DRP1a}
(\rho,u,v,p)({x},{y},0)=
\begin{cases}(1,0,0,1),& x>0.5,y>0.5,\\
  (0.5197,-0.7259,0,0.4),&    x<0.5,y>0.5,\\
  (1,-0.7259,-0.7259,1),&      x<0.5,y<0.5,\\
  (0.5197,0,-0.7259,0.4),&    x>0.5,y<0.5,
  \end{cases}
\end{equation}
which are about the interaction of four rarefaction waves.

In the first configuration, we take $\Gamma=1.4$ and assume that the initial data $-0.7259$ of fluid velocity in \eqref{eq:2DRP1a} are perturbed to $-0.7259+0.1 \xi$. The gPC-SG  method is used to study the effect of this random inputs on the flow structure. Fig.~\ref{fig:2DRP1a} gives the contours of numerical mean and standard deviation of the density at time $t=0.2$ by using the gPC-SG method with $M=3$ and $250\times 250$ uniform cells, while the reference solutions given by a collocation method with $400 \times 400$ uniform cells are also displayed. Here and following, the collocation method takes 40 Gaussian points as collocation points, and
solves the deterministic problem for each collocation point by using the fifth-order accurate finite difference WENO scheme and the third-order Runge Kutta method \eqref{eq:RK1} for time discretization.
It can be seen that the mean and standard deviation of the density are correctly captured by the gPC-SG method. For a further comparison,
the mean and standard deviation of the density are plotted along the line $y=x$, see
Fig. \ref{fig:2DRP1a_Comparison}. Those plots validate the above observation.

The second configuration takes the certain initial data \eqref{eq:2DRP1a} with uncertain adiabatic index $\Gamma$, which satisfies
$$
\Gamma (\xi) = 1.4+0.1 \xi.
$$
To analyze the effect of this uncertainty, the gPC-SG method is used to compute the numerical solutions with $M=3$ and $250\times 250$ uniform cells.
The contours of numerical mean and standard deviation of the density at time $t=0.2$ are displayed in Fig.~\ref{fig:2DRP1b}, where the reference solutions are obtained by the collocation method with $400 \times 400$ uniform cells.
It is seen that the mean of the density is similar to that in the first configuration, but the standard deviation is very different.
Moreover, the gPC-SG method captures the flow structure and resolves the standard deviation with high resolution.
Fig. \ref{fig:2DRP1b_Comparison} gives a further comparison of the the mean and standard deviation of the  density  along the line $y=x$,
and demonstrates good agreement between the numerical solutions of the gPC-SG method and the reference ones.

\begin{figure}[htbp]
  \centering
  {\includegraphics[width=0.48\textwidth]{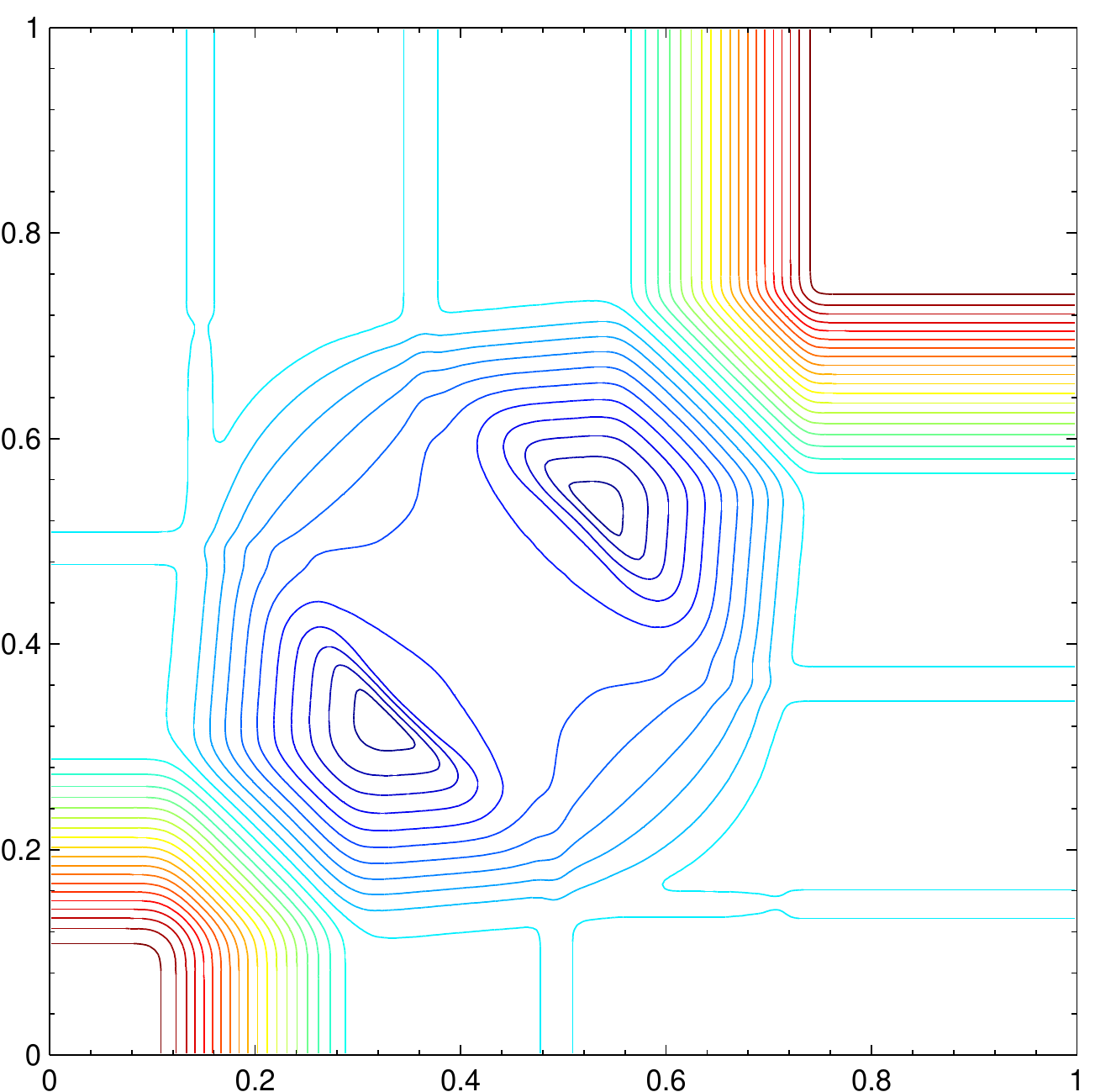}}
  {\includegraphics[width=0.48\textwidth]{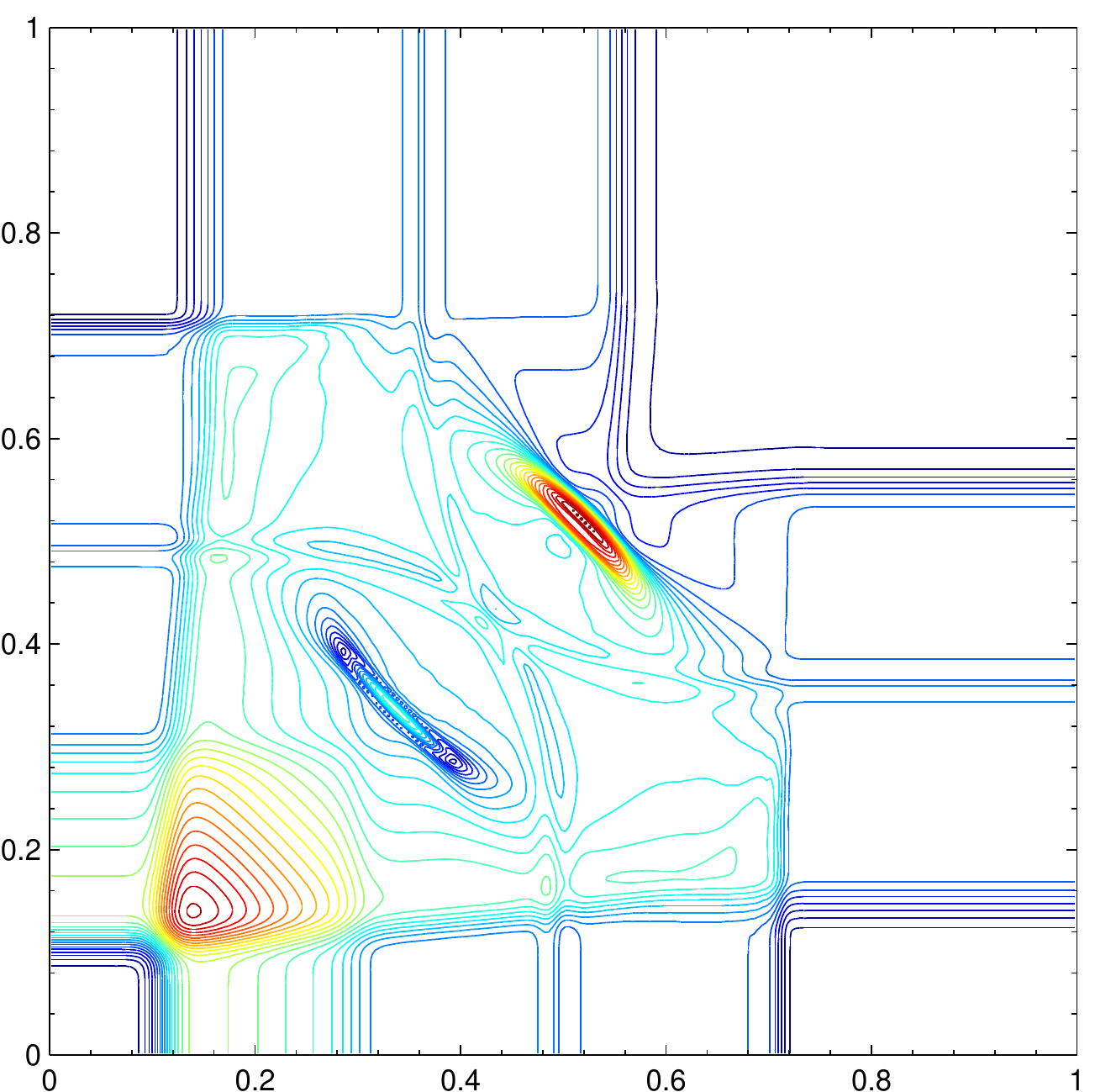}}
  {\includegraphics[width=0.48\textwidth]{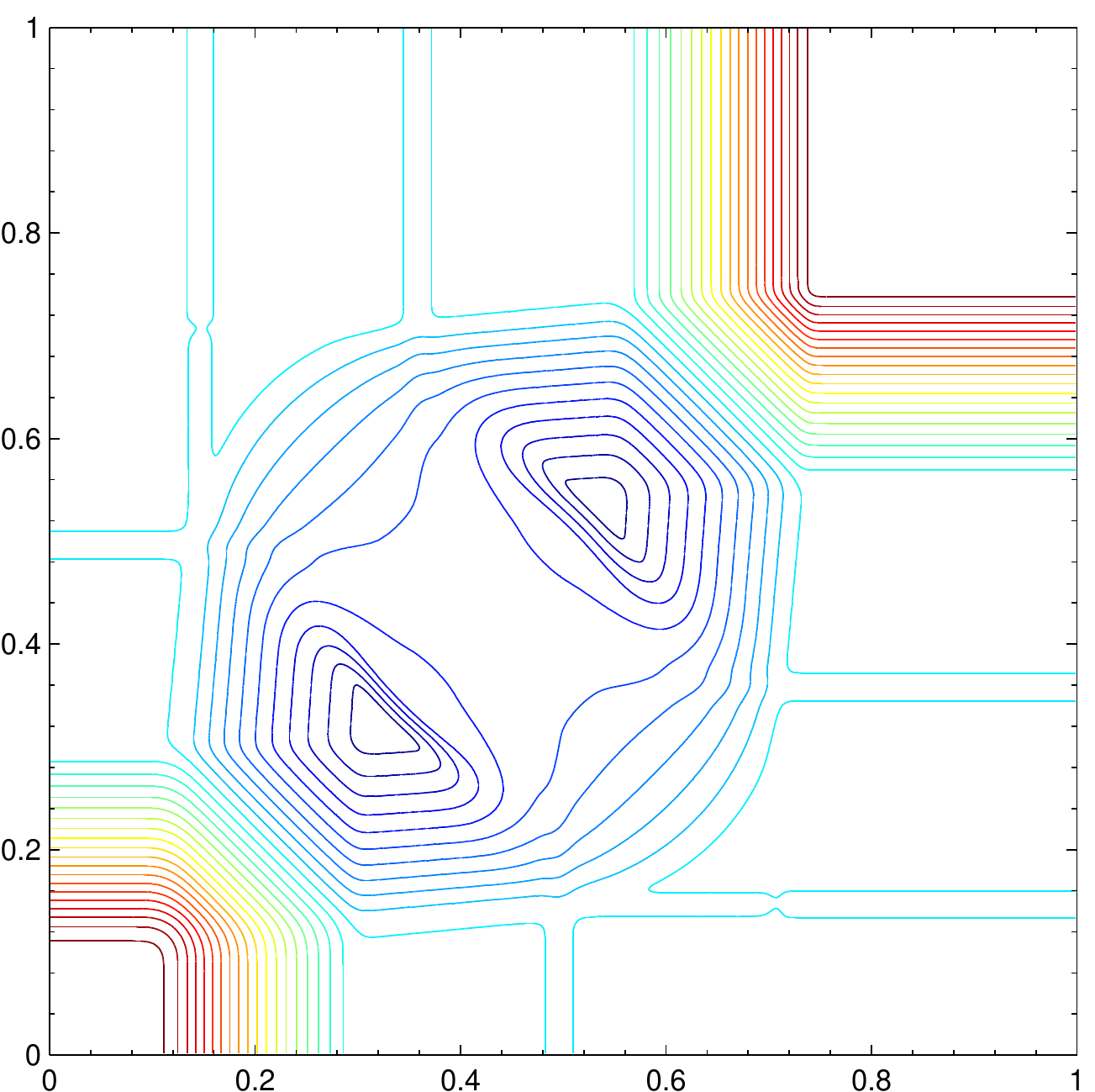}}
  {\includegraphics[width=0.48\textwidth]{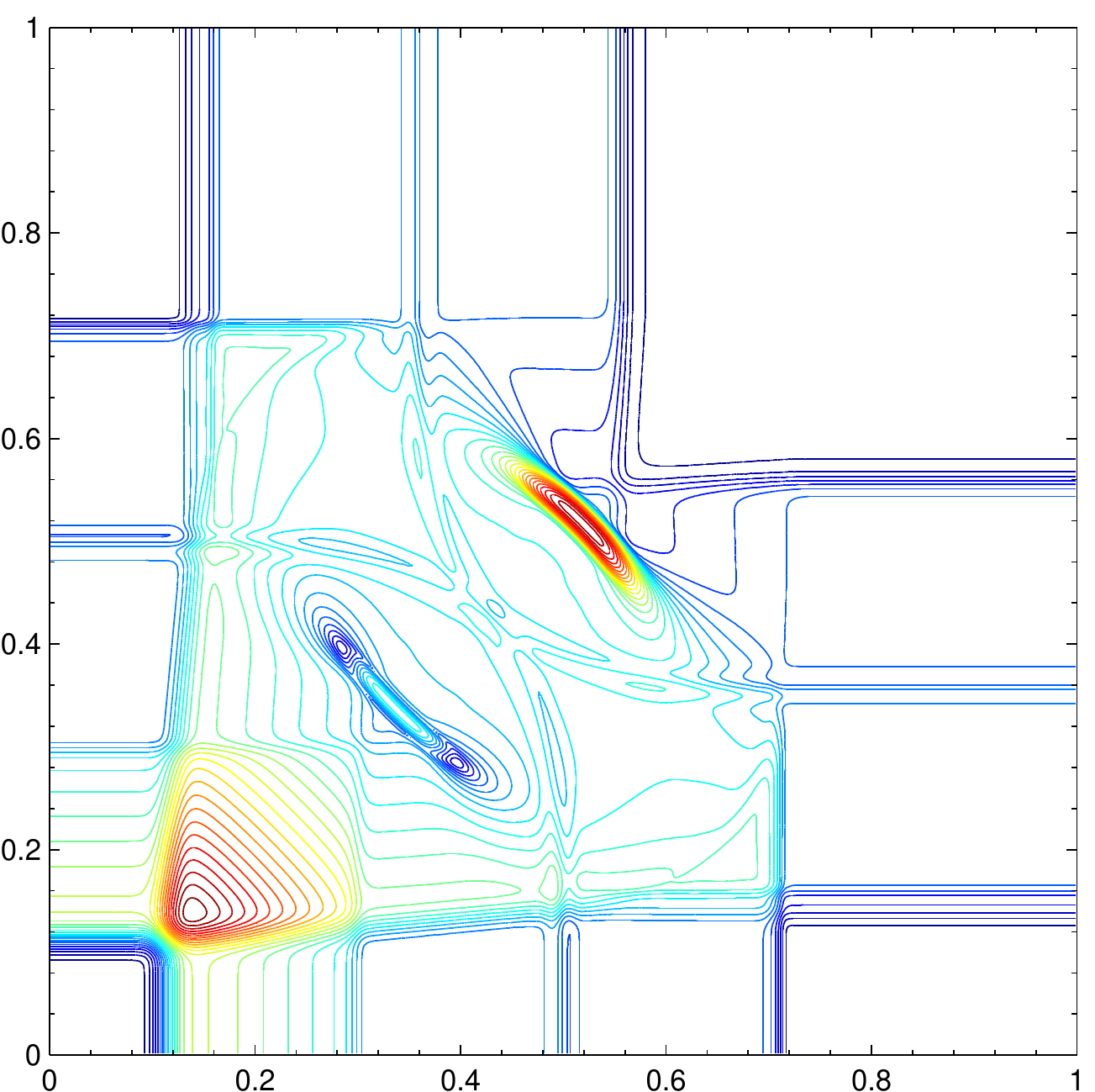}}
  \caption{\small The first configuration of Example \ref{example:RP1}:
  The contours of numerical  mean (left) and standard deviation (right) of  the density
  at $t=0.2$ within the domain $[0,1] \times [0,1]$.
         30 equally spaced contour lines are used. From top to bottom:  gPC-SG method with $M=3$ and $250 \times 250$ uniform cells, and
         collocation method with $400 \times 400$ uniform cells.
 }
  \label{fig:2DRP1a}
\end{figure}

\begin{figure}[htbp]
  \centering
  \includegraphics[width=0.48\textwidth]{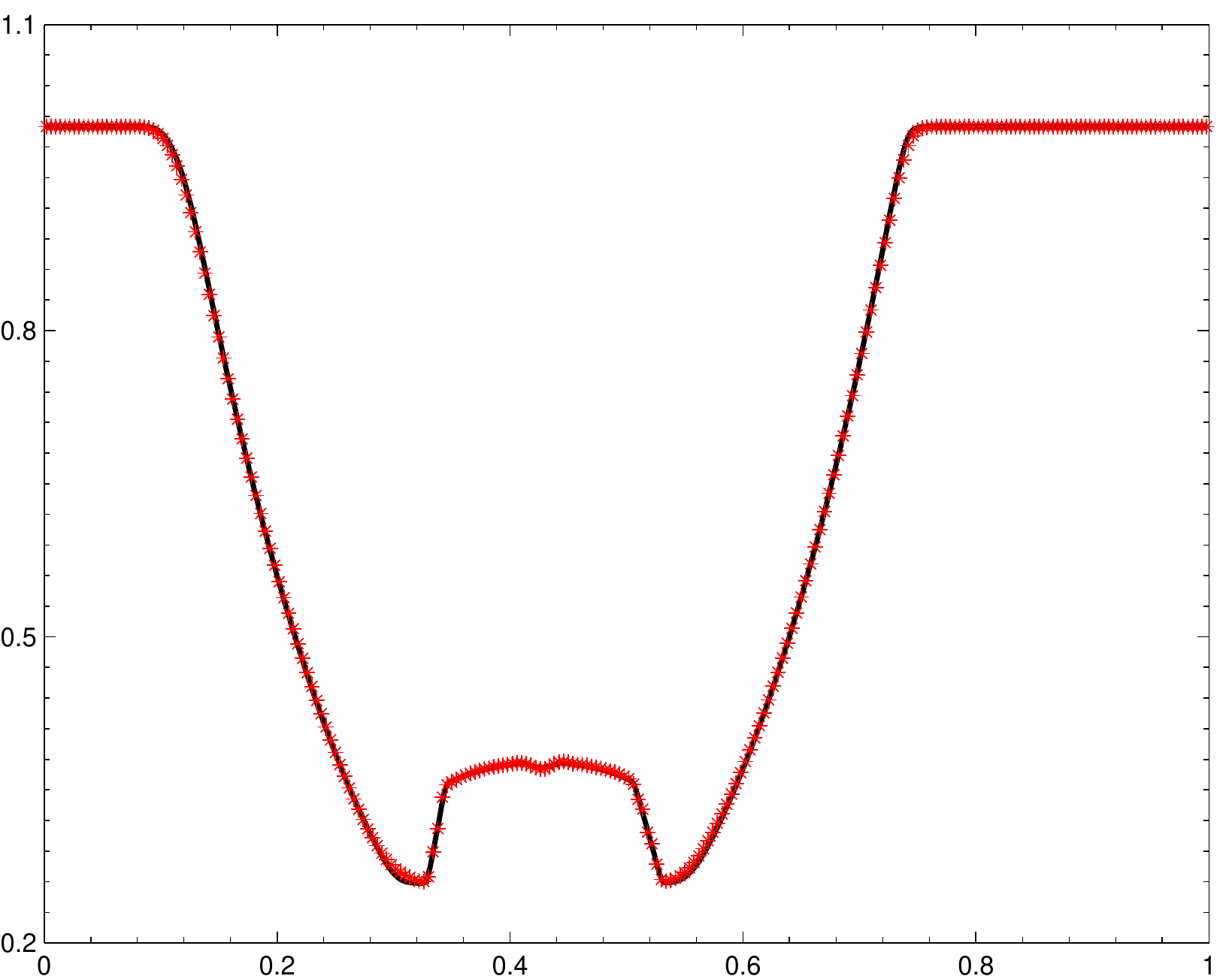}
  \includegraphics[width=0.48\textwidth]{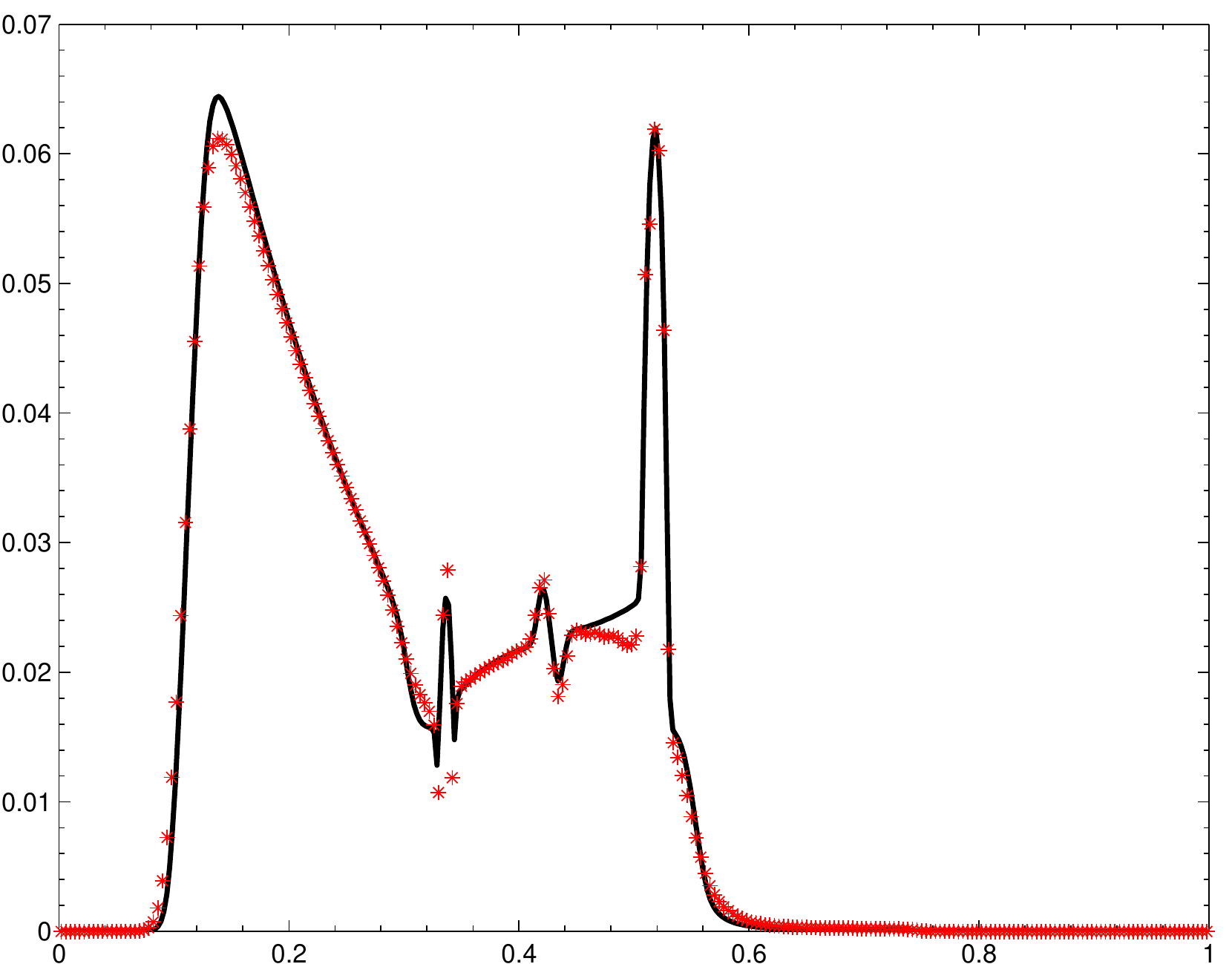}
  \caption{\small Same as Fig. \ref{fig:2DRP1a}, except for mean (left) and standard deviation (right) of the density
   along the line $y=x$  within the scaled interval $[0,1]$. ``${\color{red}*}$'' denotes the numerical results given by gPC-SG method with $250 \times 250$ uniform cells, while the solid lines represent the numerical results of
         collocation method $400 \times 400$ uniform cells.
 }
  \label{fig:2DRP1a_Comparison}
\end{figure}

\begin{figure}[htbp]
  \centering
  {\includegraphics[width=0.48\textwidth]{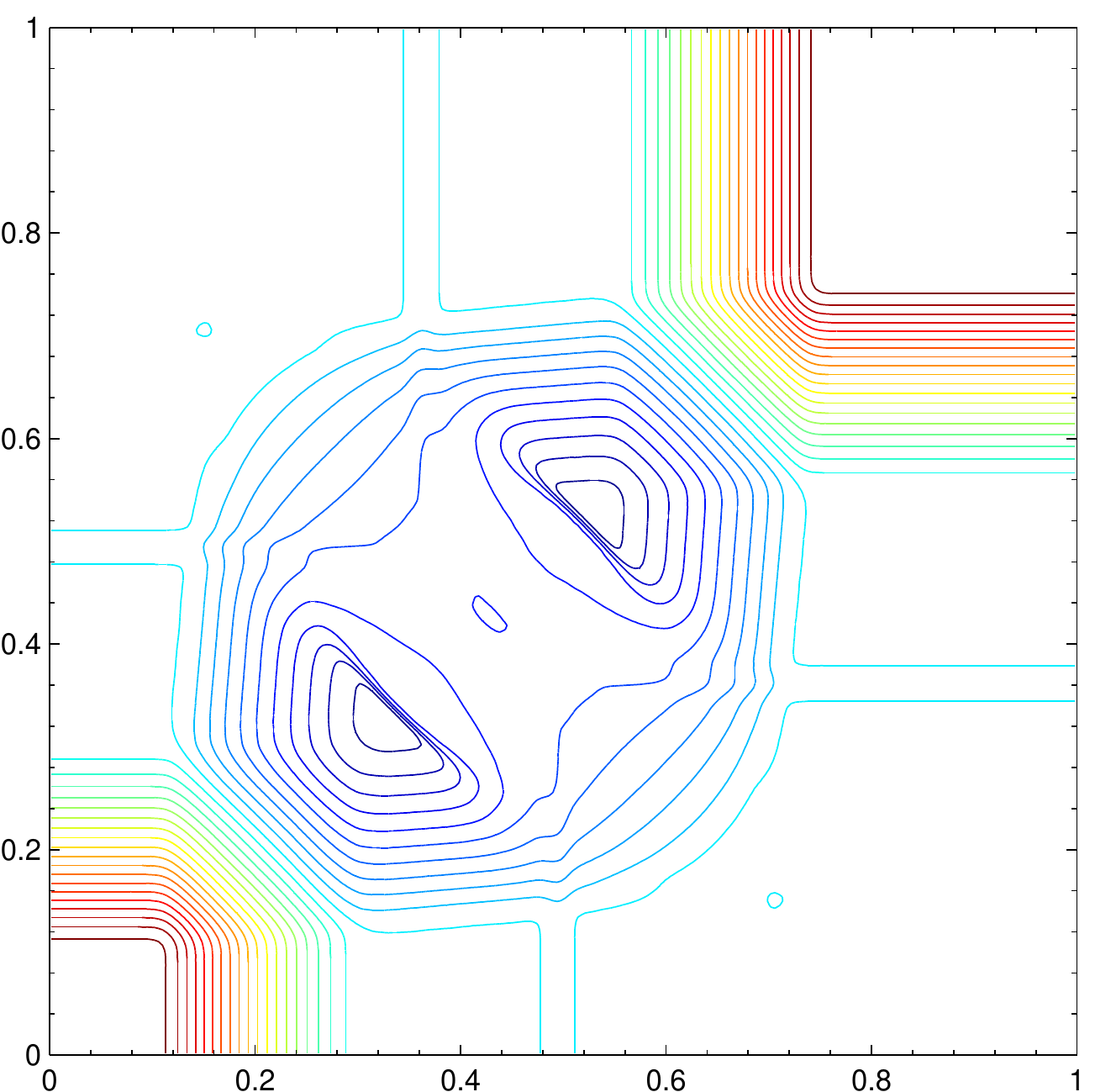}}
  {\includegraphics[width=0.48\textwidth]{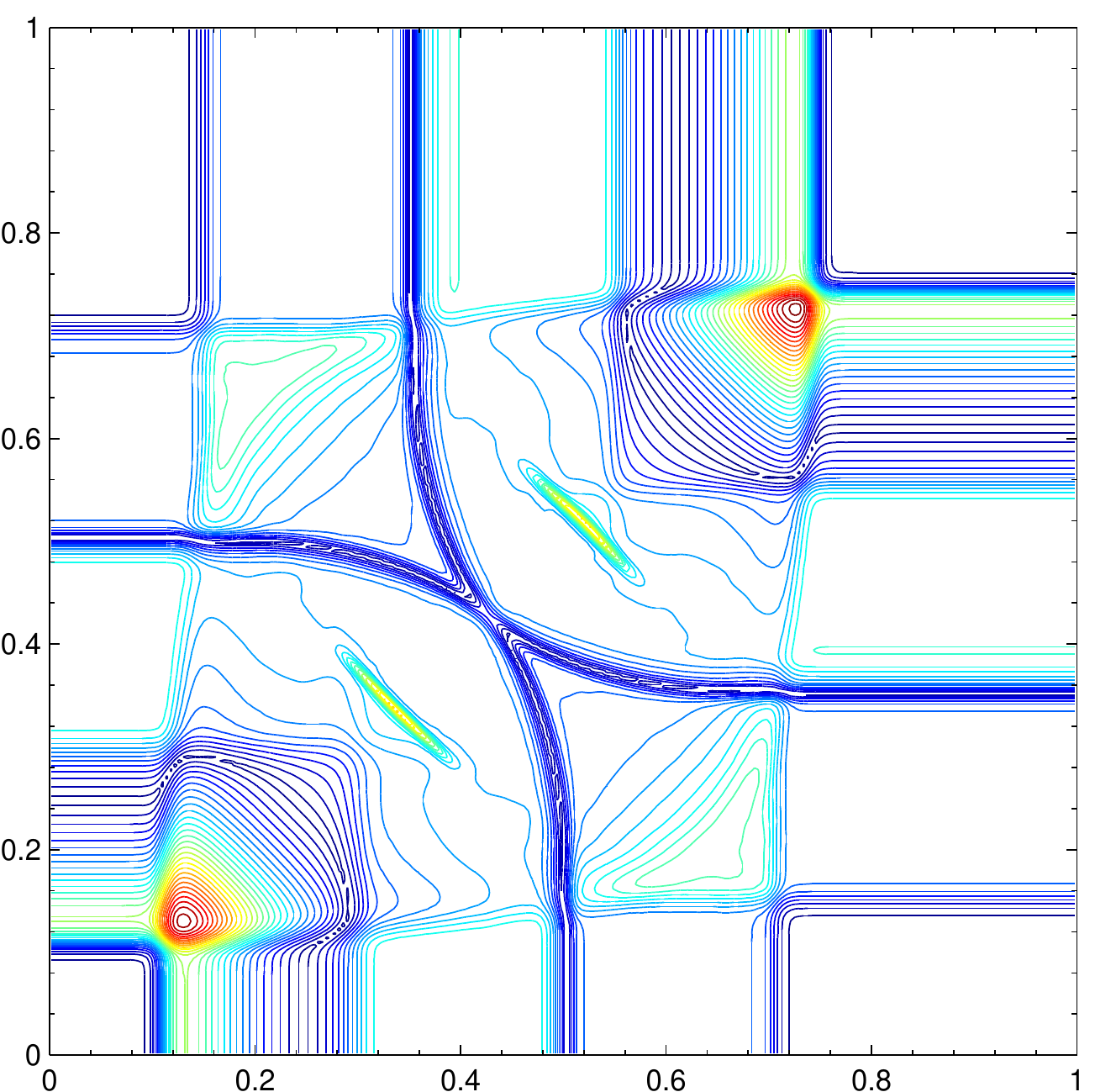}}
  {\includegraphics[width=0.48\textwidth]{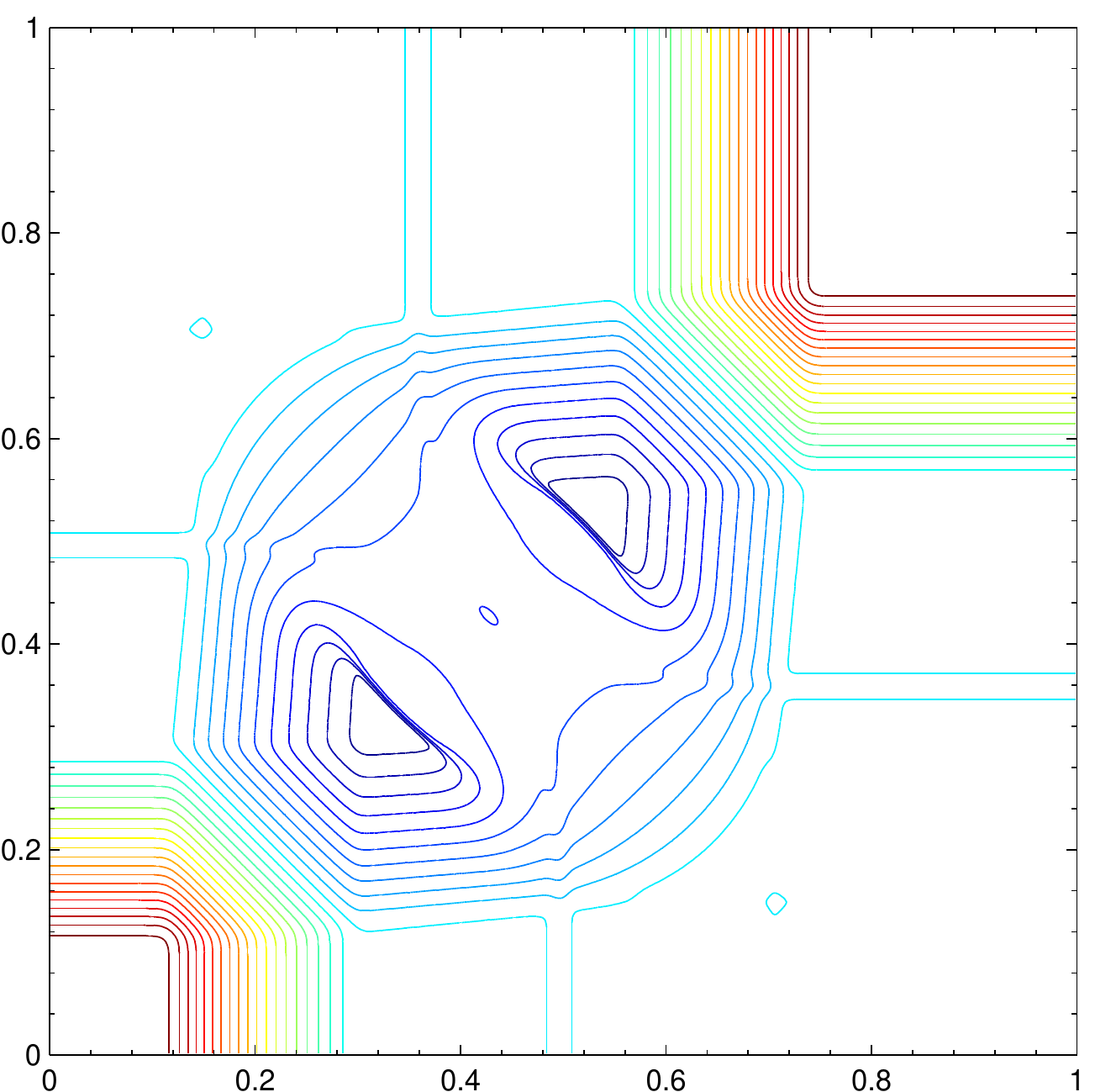}}
  {\includegraphics[width=0.48\textwidth]{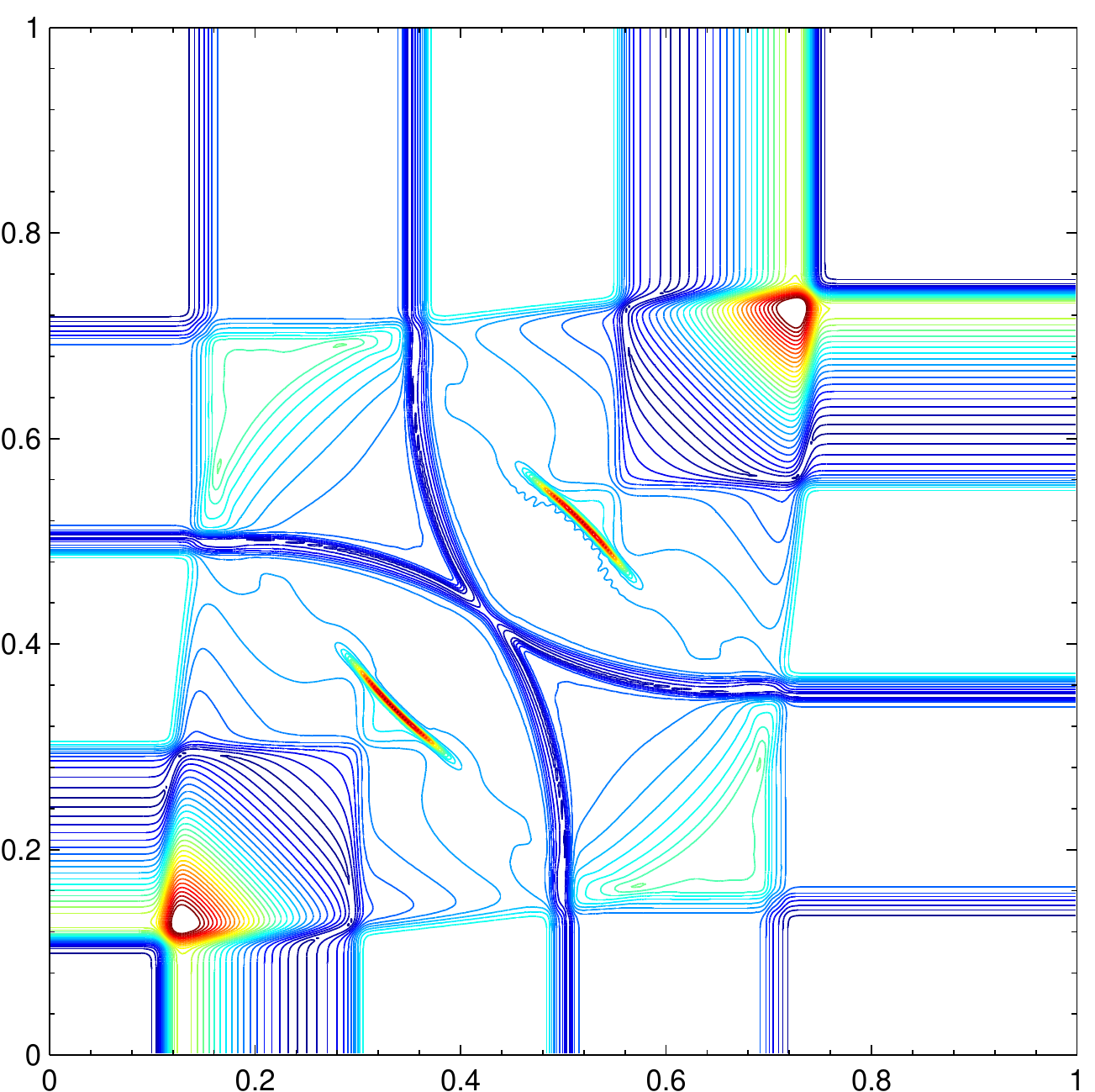}}
  \caption{\small Same as Fig. \ref{fig:2DRP1a}, except for the second configuration of Example \ref{example:RP1}.
 }
  \label{fig:2DRP1b}
\end{figure}

\begin{figure}[htbp]
  \centering
  \includegraphics[width=0.48\textwidth]{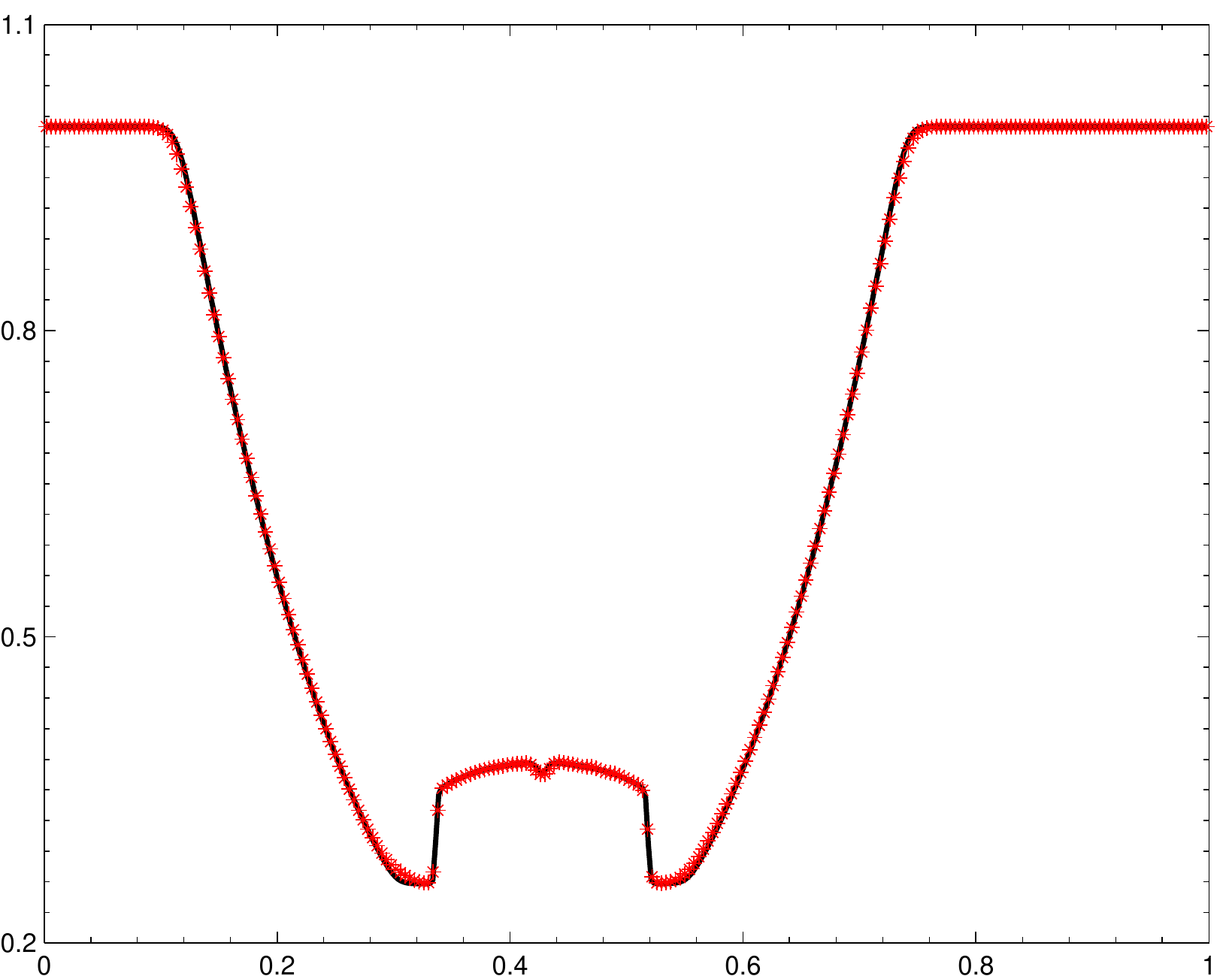}
  \includegraphics[width=0.48\textwidth]{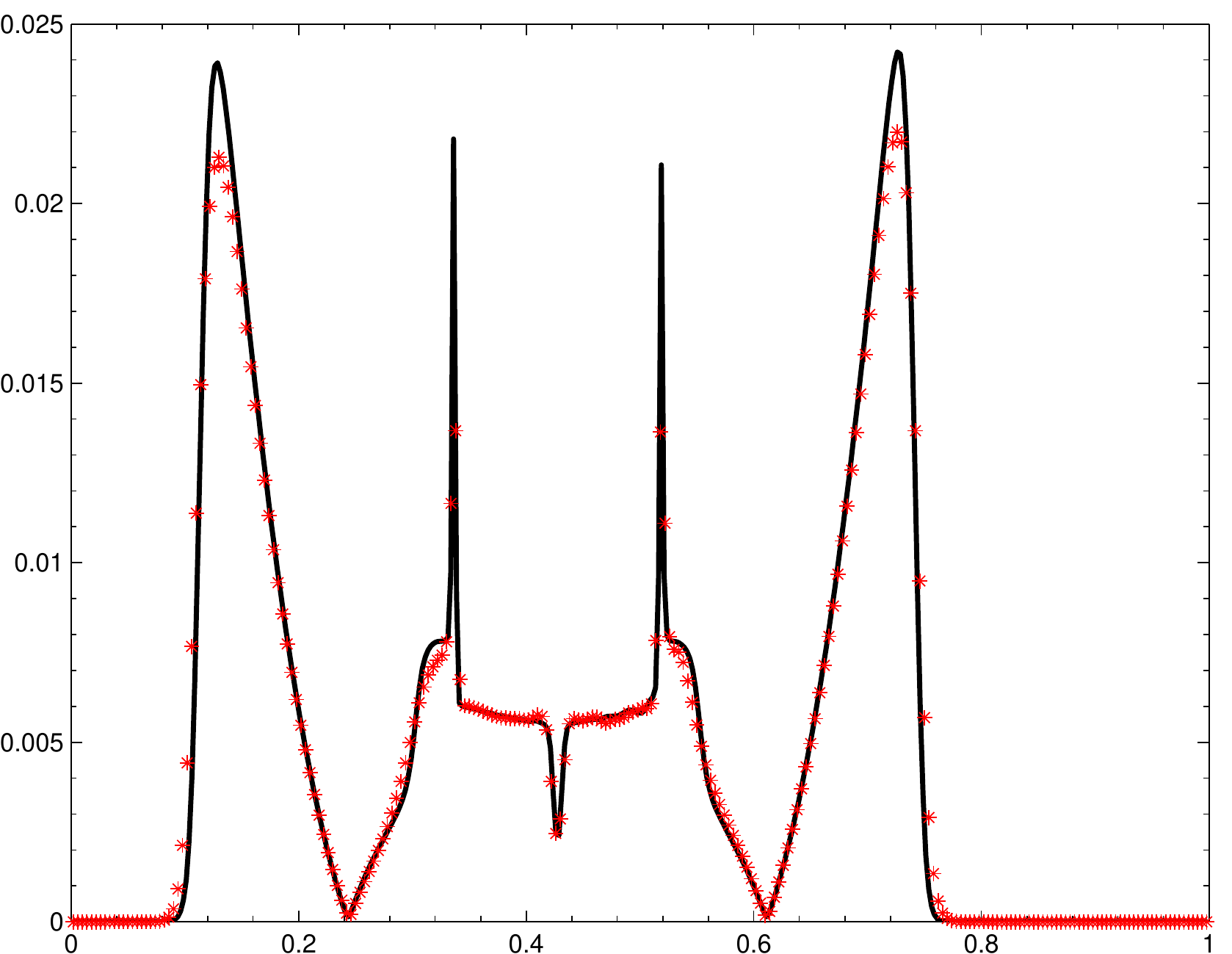}
  \caption{\small Same as Fig. \ref{fig:2DRP1a_Comparison}, except for the second configuration of Example \ref{example:RP1}.
 }
  \label{fig:2DRP1b_Comparison}
\end{figure}

\end{example}

\begin{example}[2D Riemann problems III and IV]\label{example:RP2}\rm
The last two problems are two perturbed versions of a classic deterministic Riemann problem with initial conditions \cite{LaxLiu1998}
\begin{equation}\label{eq:2DRP2a}
(\rho,u,v,p)({x},{y},0)=
\begin{cases}(0.5197,0.1,0.1,0.1,0.4),& x>0.5,y>0.5,\\
  (1,-0.6259,0.1,1),&    x<0.5,y>0.5,\\
  (0.8,0.1,0.1,1),&      x<0.5,y<0.5,\\
  (1,0.1,-0.6259,1),&    x>0.5,y<0.5,
  \end{cases}
\end{equation}
which describe the interaction of two rarefaction waves and two contact discontinuities.

The first  case takes $\Gamma=1.4$ and assumes that the initial density $\rho(x,y,0)$ given in \eqref{eq:2DRP2a} contains ``ten percent'' uncertainty, that is,
it is perturbed to $(1+0.1 \xi)\rho(x,y,0)$. The gPC-SG method is used to study the effect of this random inputs on the flow structure.  Fig.~\ref{fig:2DRP2a} displays the contours of numerical mean and standard deviation of the density at time $t=0.2$ by using the gPC-SG method with $M=3$ and $250\times 250$ uniform cells, and the reference ones given by the collocation method with $400 \times 400$ uniform cells.
Fig. \ref{fig:2DRP2a_Comparison} gives the mean and standard deviation of the  density    along the line $y=x$.
We see that the results given by the gPC-SG method agree well with the reference solutions.

\begin{figure}[htbp]
  \centering
  {\includegraphics[width=0.48\textwidth]{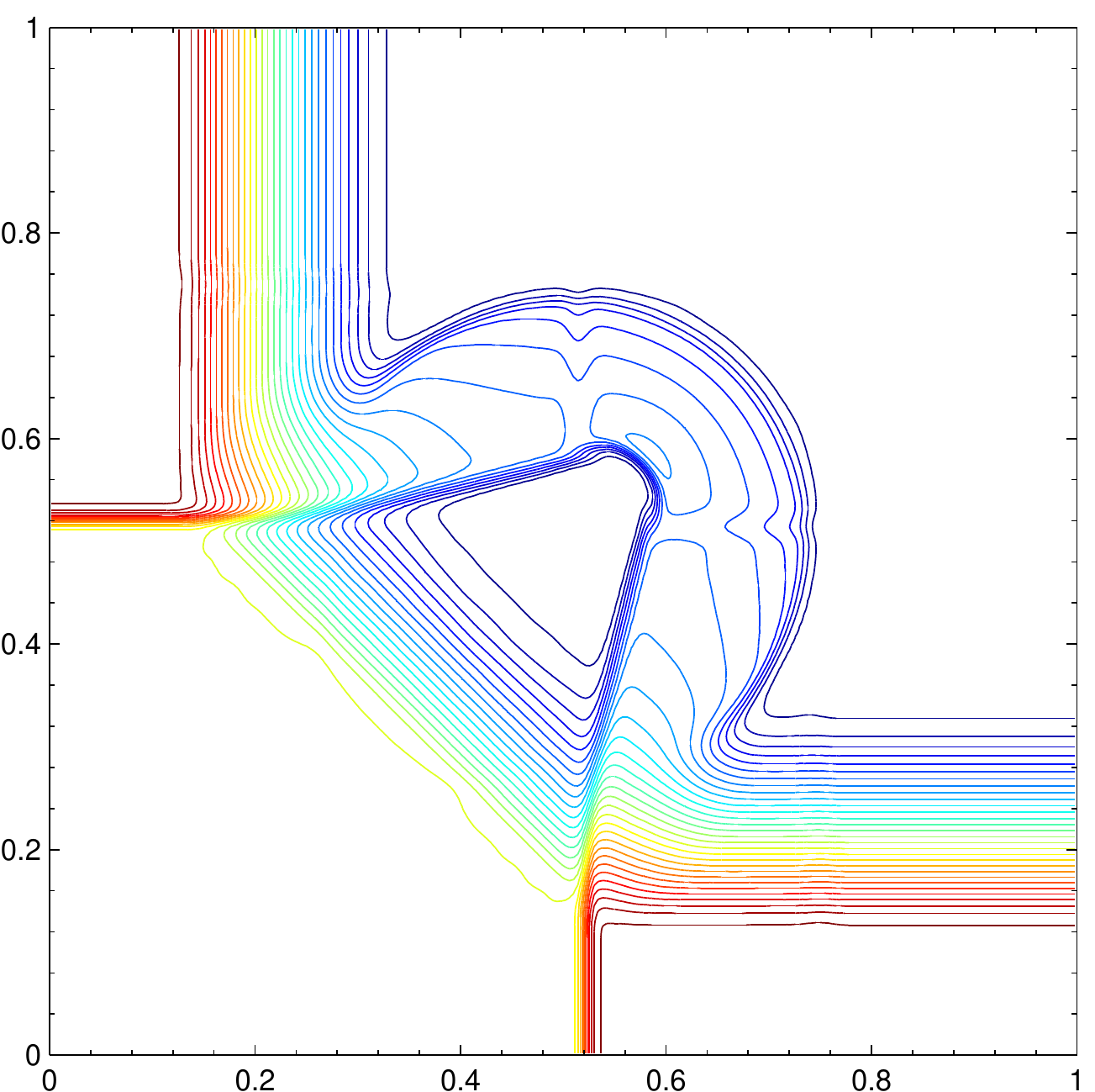}}
  {\includegraphics[width=0.48\textwidth]{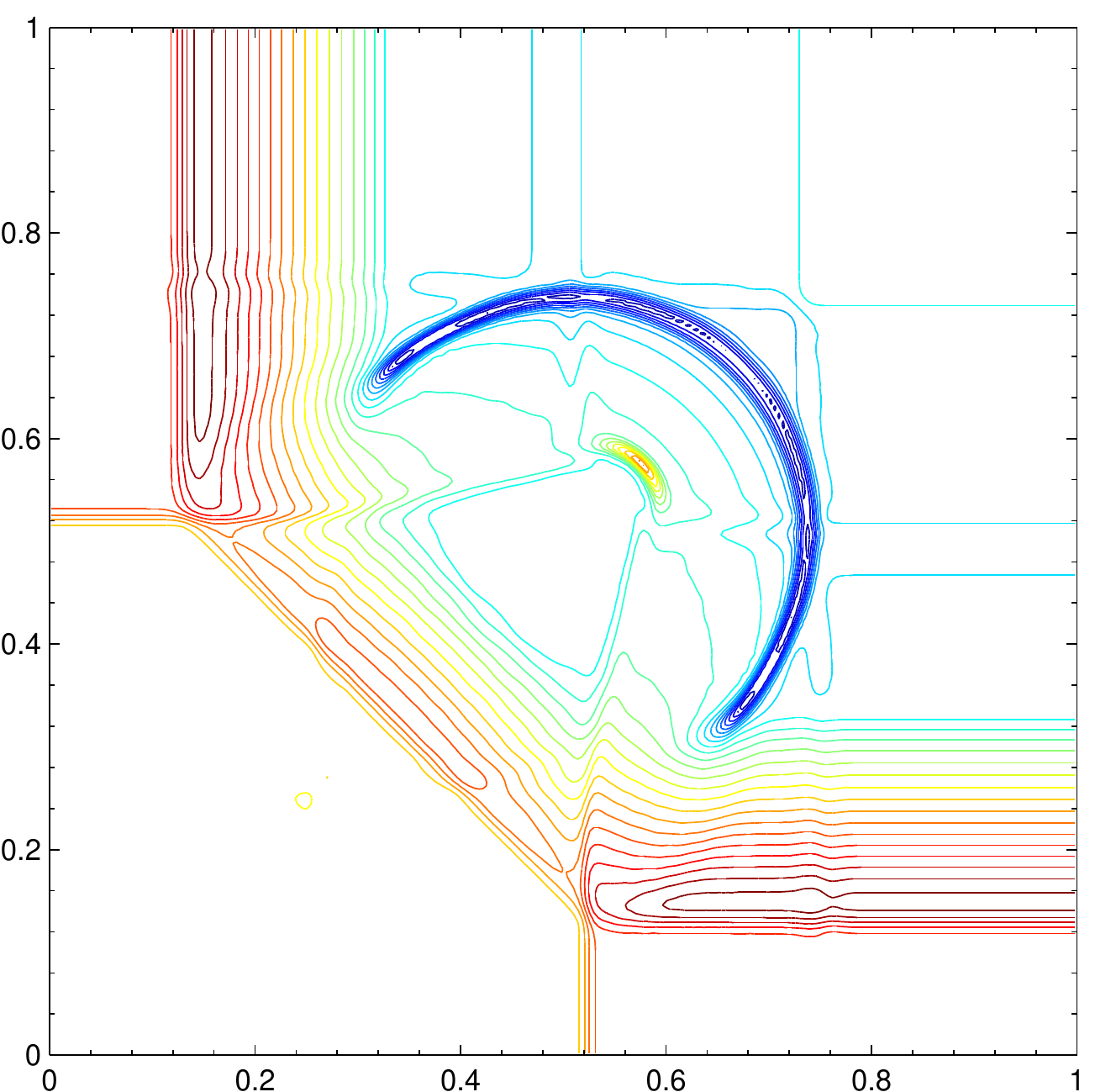}}
  {\includegraphics[width=0.48\textwidth]{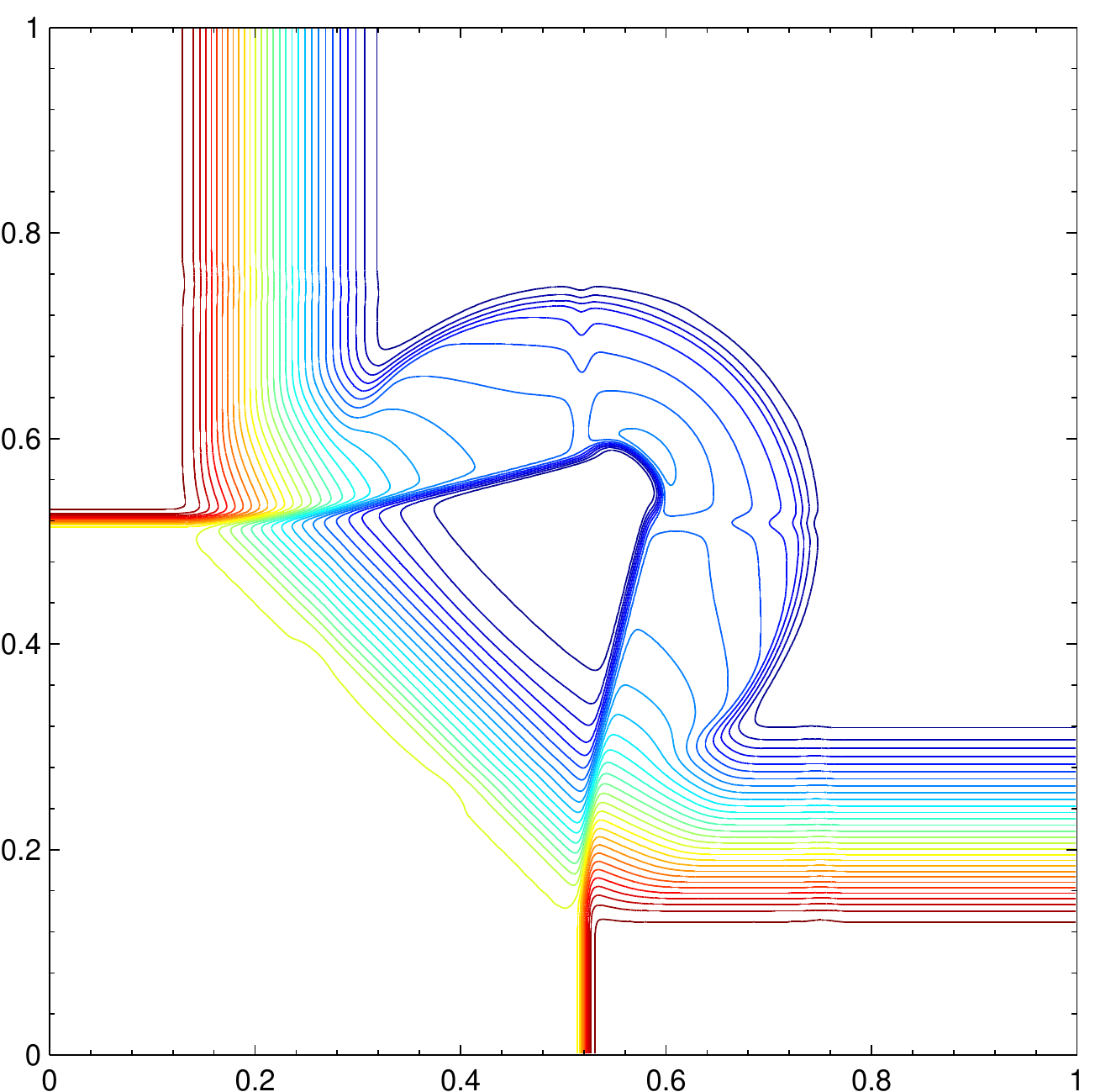}}
  {\includegraphics[width=0.48\textwidth]{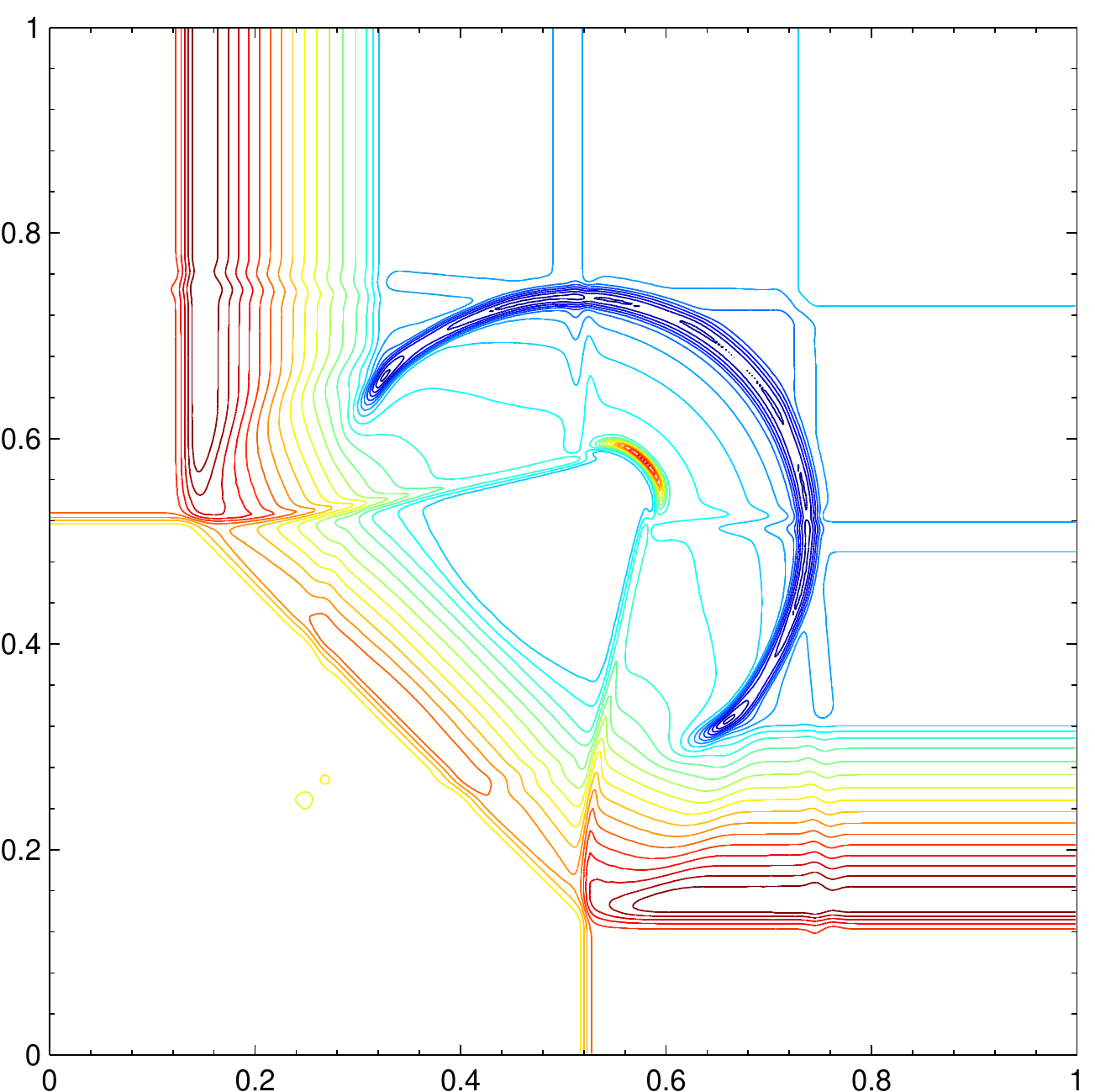}}
  \caption{\small Same as Fig. \ref{fig:2DRP1a}, except for the first configuration of Example \ref{example:RP2} and 25 equally spaced contour lines.
 }
  \label{fig:2DRP2a}
\end{figure}

\begin{figure}[htbp]
  \centering
  \includegraphics[width=0.48\textwidth]{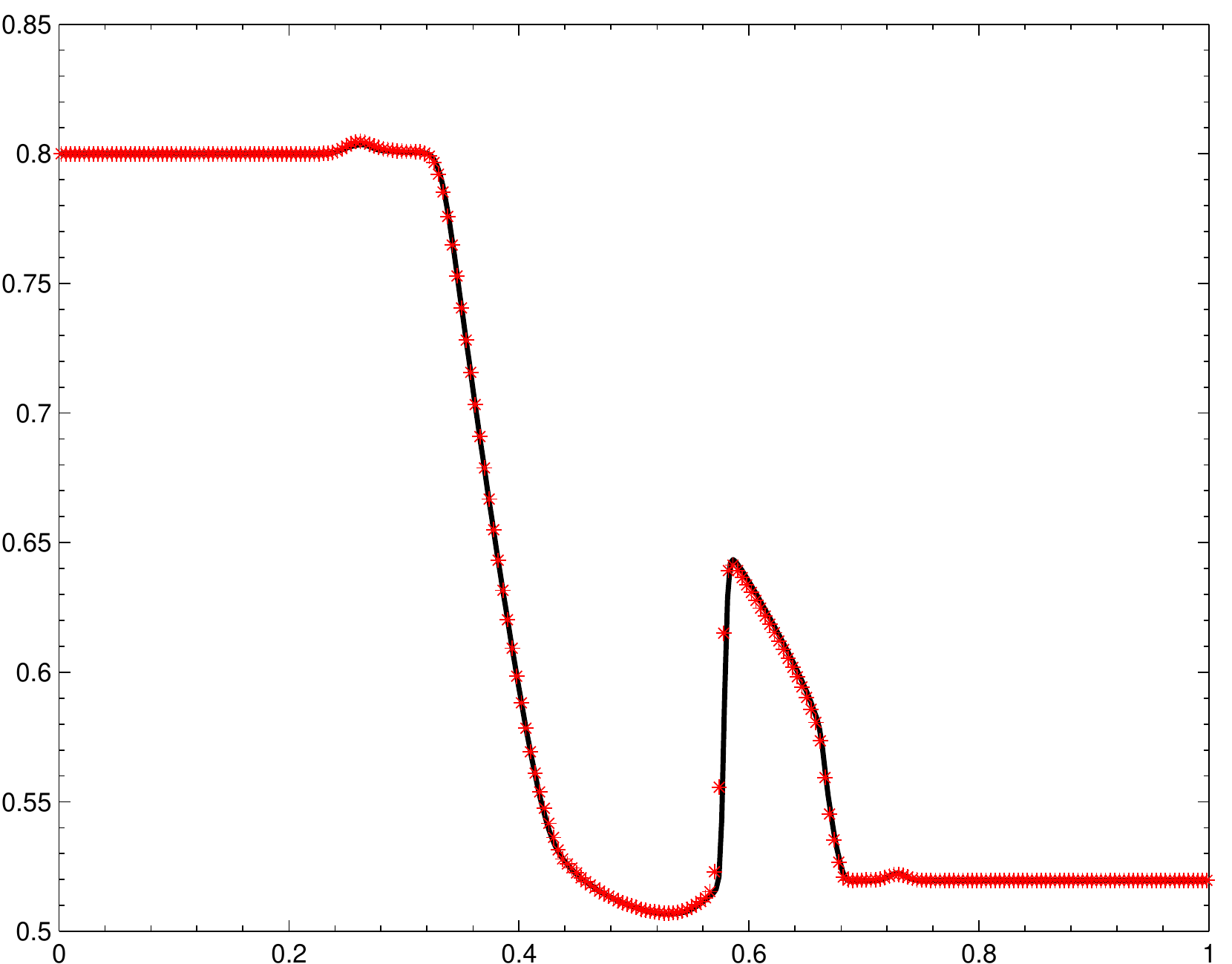}
  \includegraphics[width=0.48\textwidth]{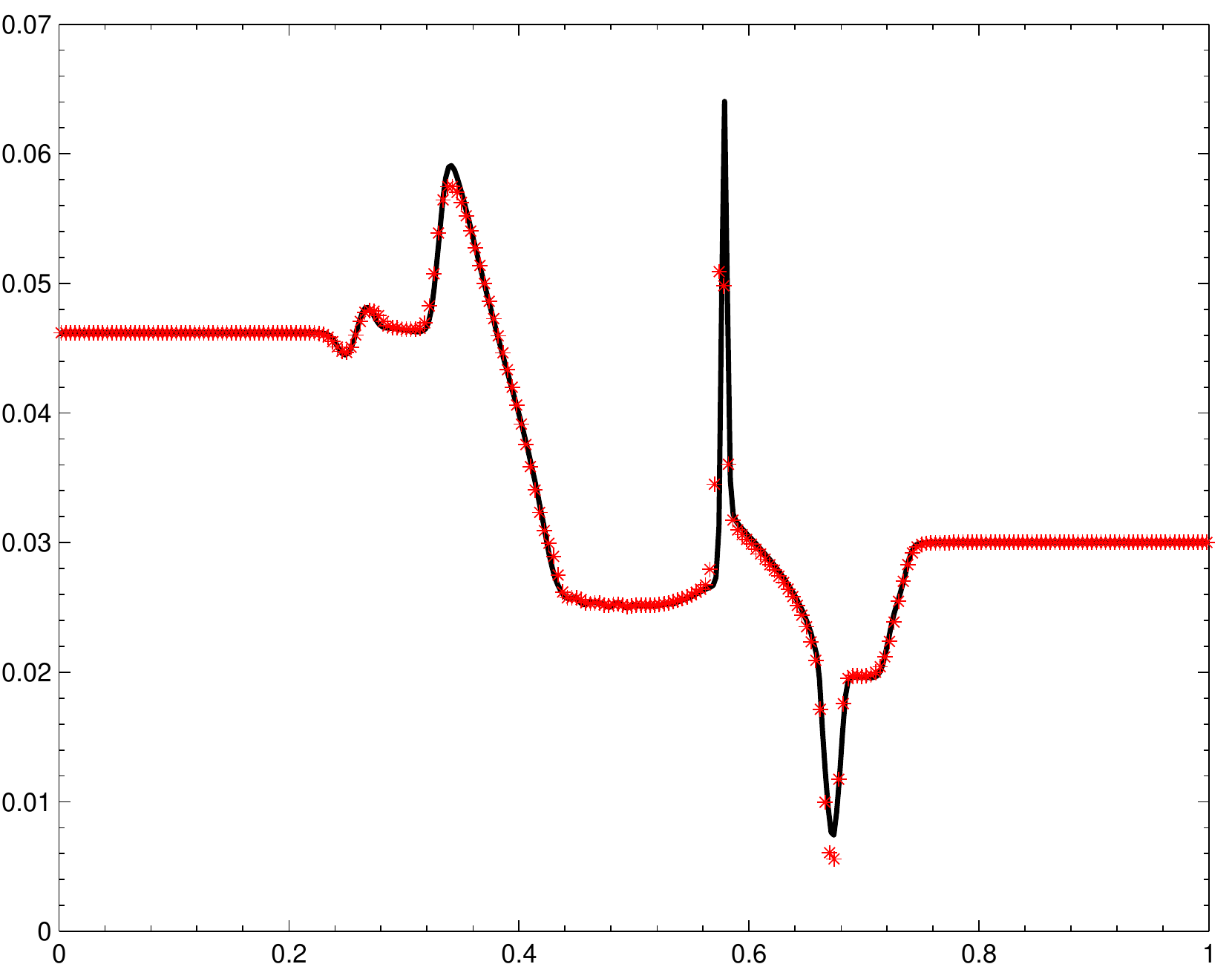}
  \caption{\small Same as Fig. \ref{fig:2DRP1a_Comparison}, except for the first configuration of Example \ref{example:RP2}.
 }
  \label{fig:2DRP2a_Comparison}
\end{figure}

The second case considers certain initial data \eqref{eq:2DRP2a}, and uncertain adiabatic index
$$
\Gamma (\xi) = 1.4+0.1 \xi.
$$
The contours of numerical mean and standard deviation of the density at time $t=0.2$ given by the gPC-SG method with $M=3$ and $250\times 250$ uniform cells
are displayed in Fig.~\ref{fig:2DRP2b}, where the reference solutions are given by the collocation method with $400 \times 400$ uniform cells.
The numerical results exhibit the good performance of the proposed gPC-SG method in resolving 2D flow structures and quantifying the uncertainties.
For a further comparison, the mean and standard deviation of the density are plotted along the line $y=x$, see
Fig. \ref{fig:2DRP2b_Comparison}. It can be seen clearly that the the means and standard deviations obtained
by the two methods are in good agreement.

\begin{figure}[htbp]
  \centering
  {\includegraphics[width=0.48\textwidth]{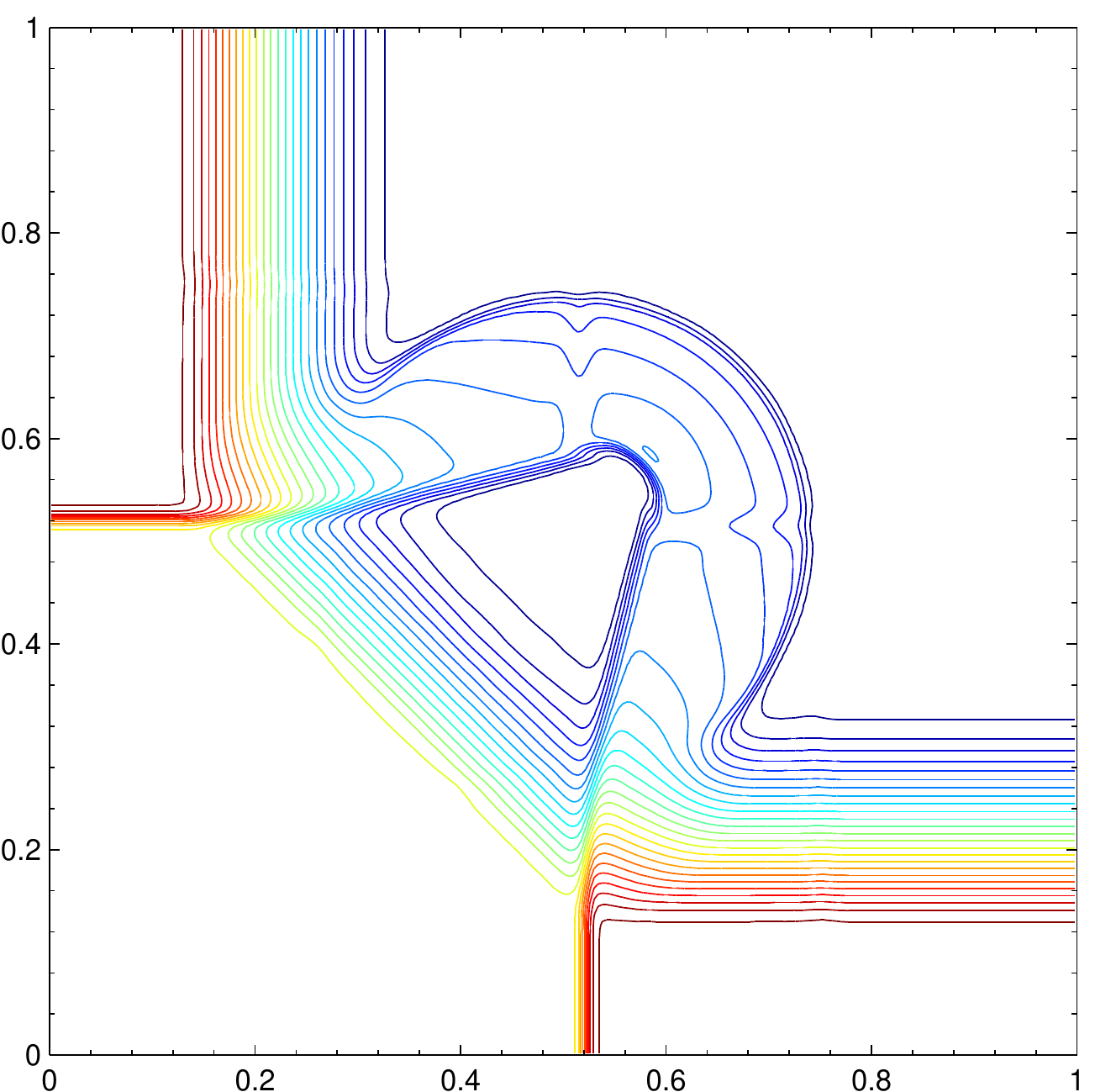}}
  {\includegraphics[width=0.48\textwidth]{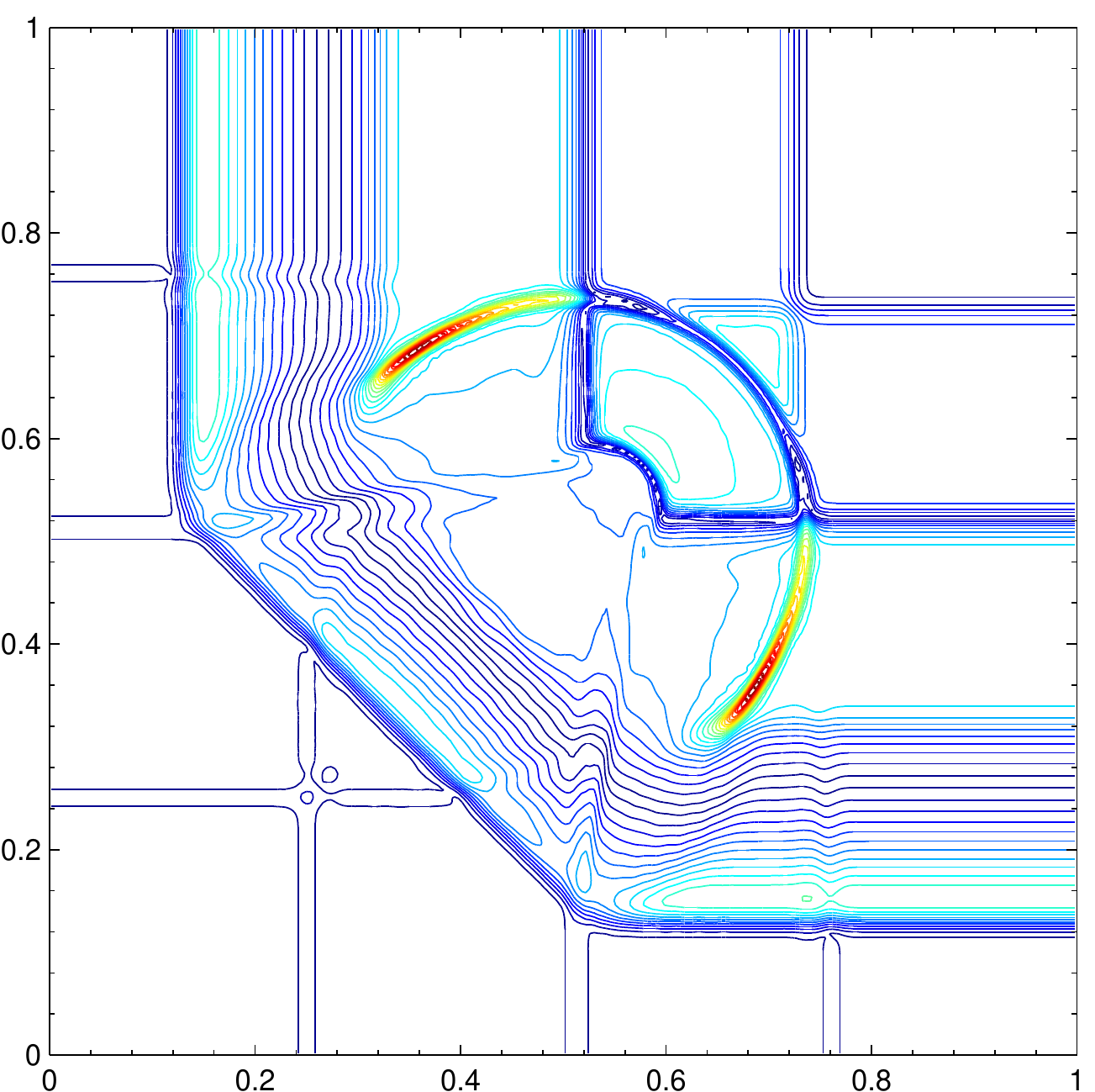}}
  {\includegraphics[width=0.48\textwidth]{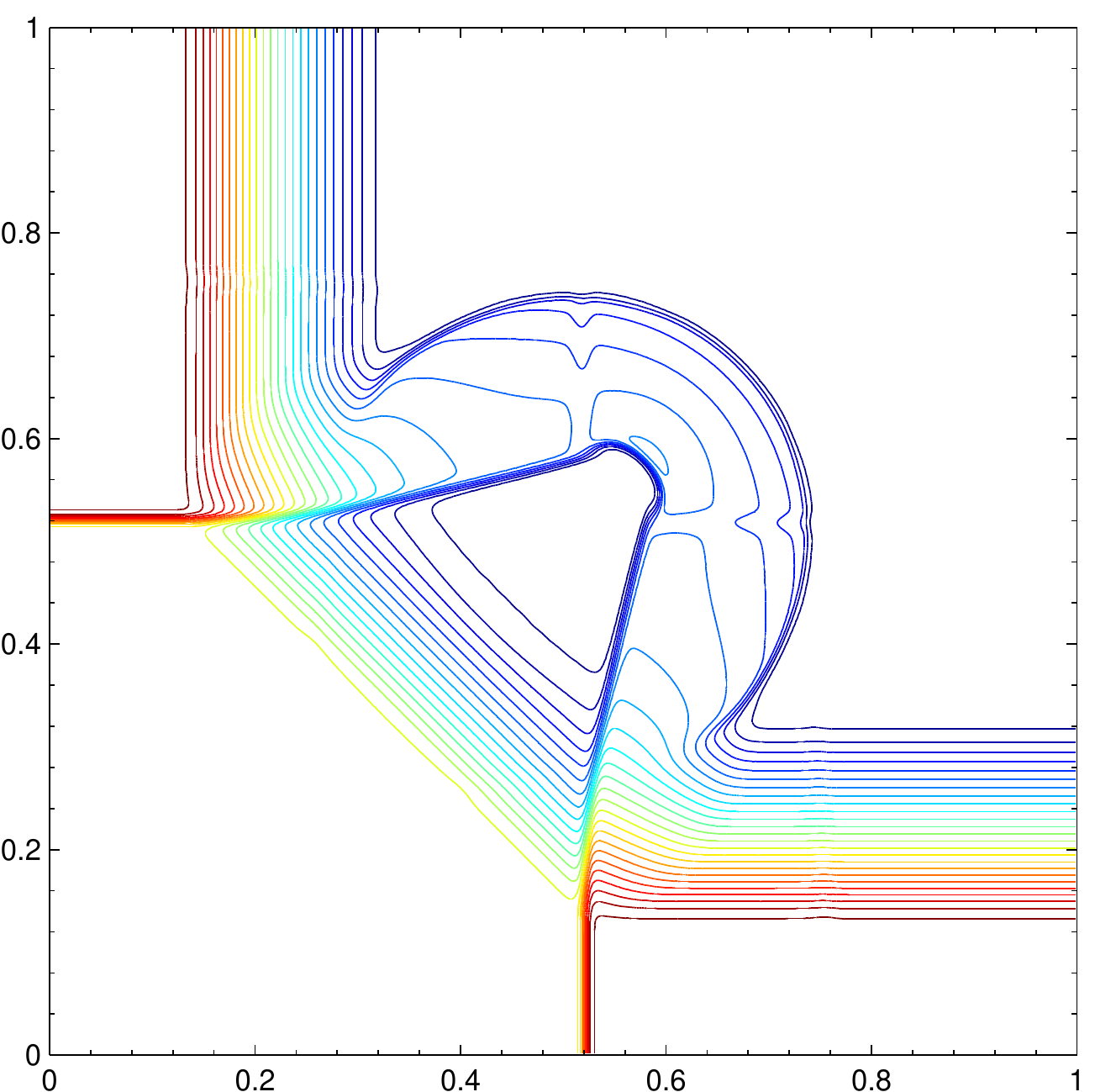}}
  {\includegraphics[width=0.48\textwidth]{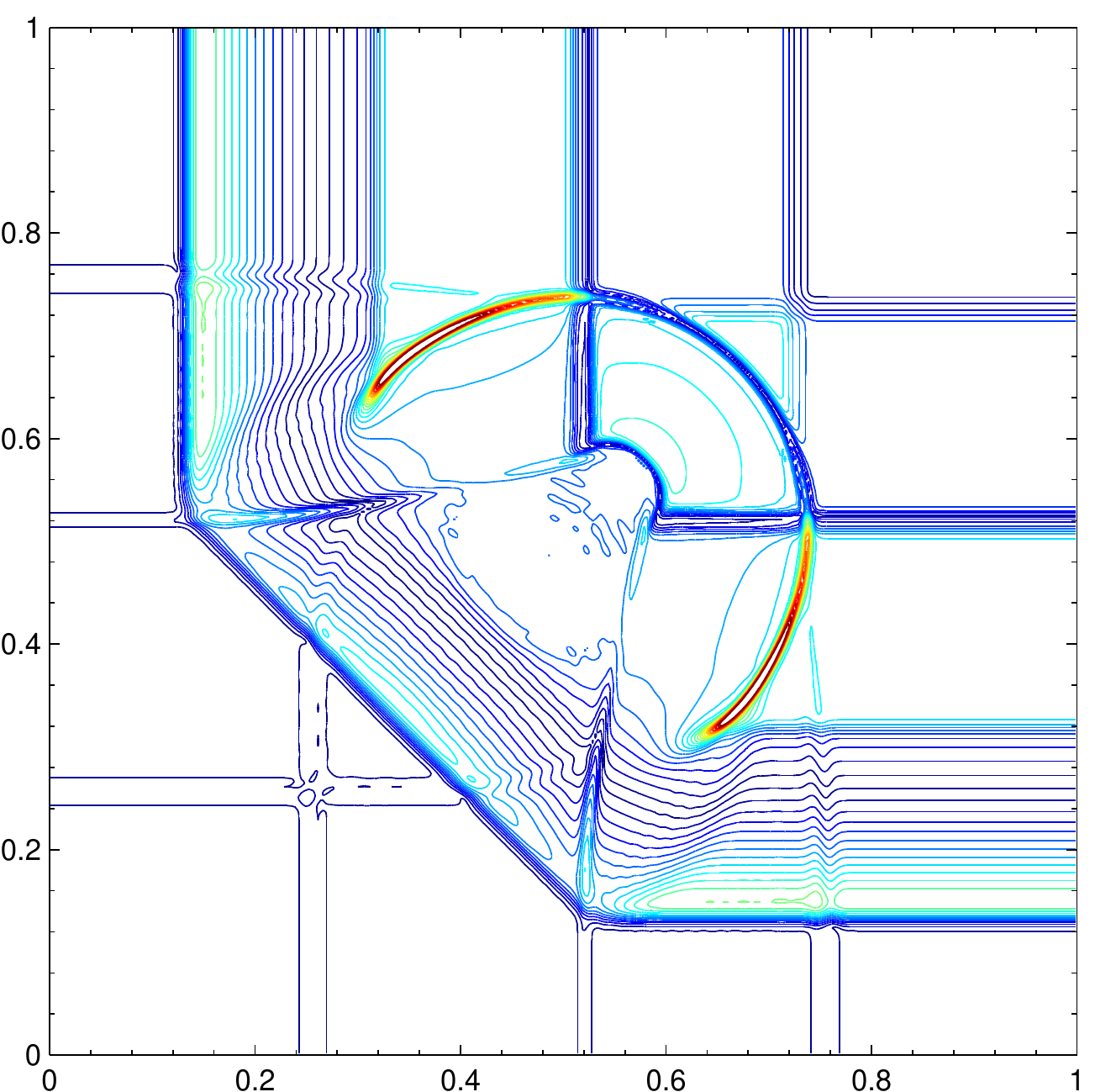}}
  \caption{\small Same as Fig. \ref{fig:2DRP1a}, except for the second configuration of Example \ref{example:RP2} and 25 equally spaced contour lines.
 }
  \label{fig:2DRP2b}
\end{figure}

\begin{figure}[htbp]
  \centering
  \includegraphics[width=0.48\textwidth]{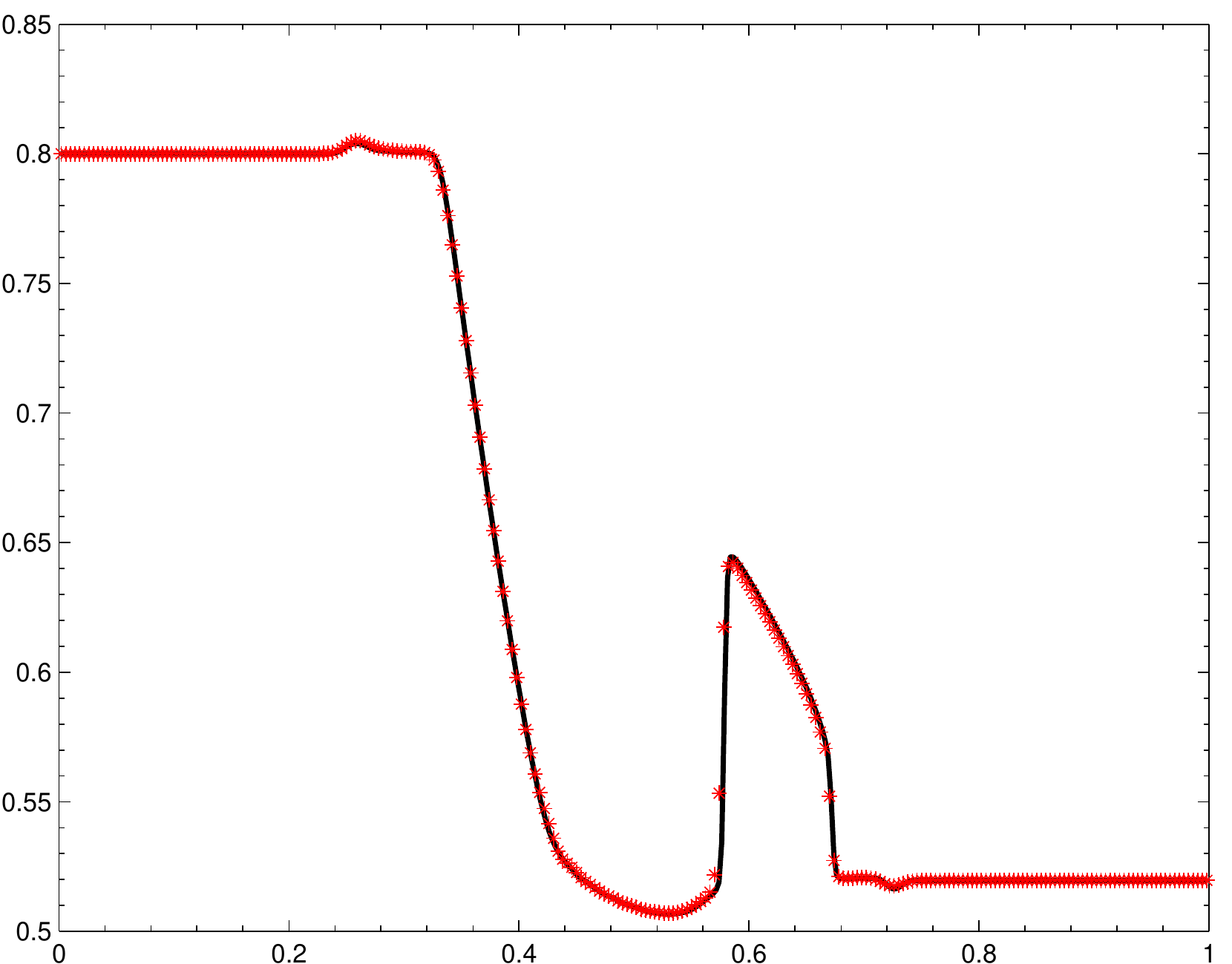}
  \includegraphics[width=0.48\textwidth]{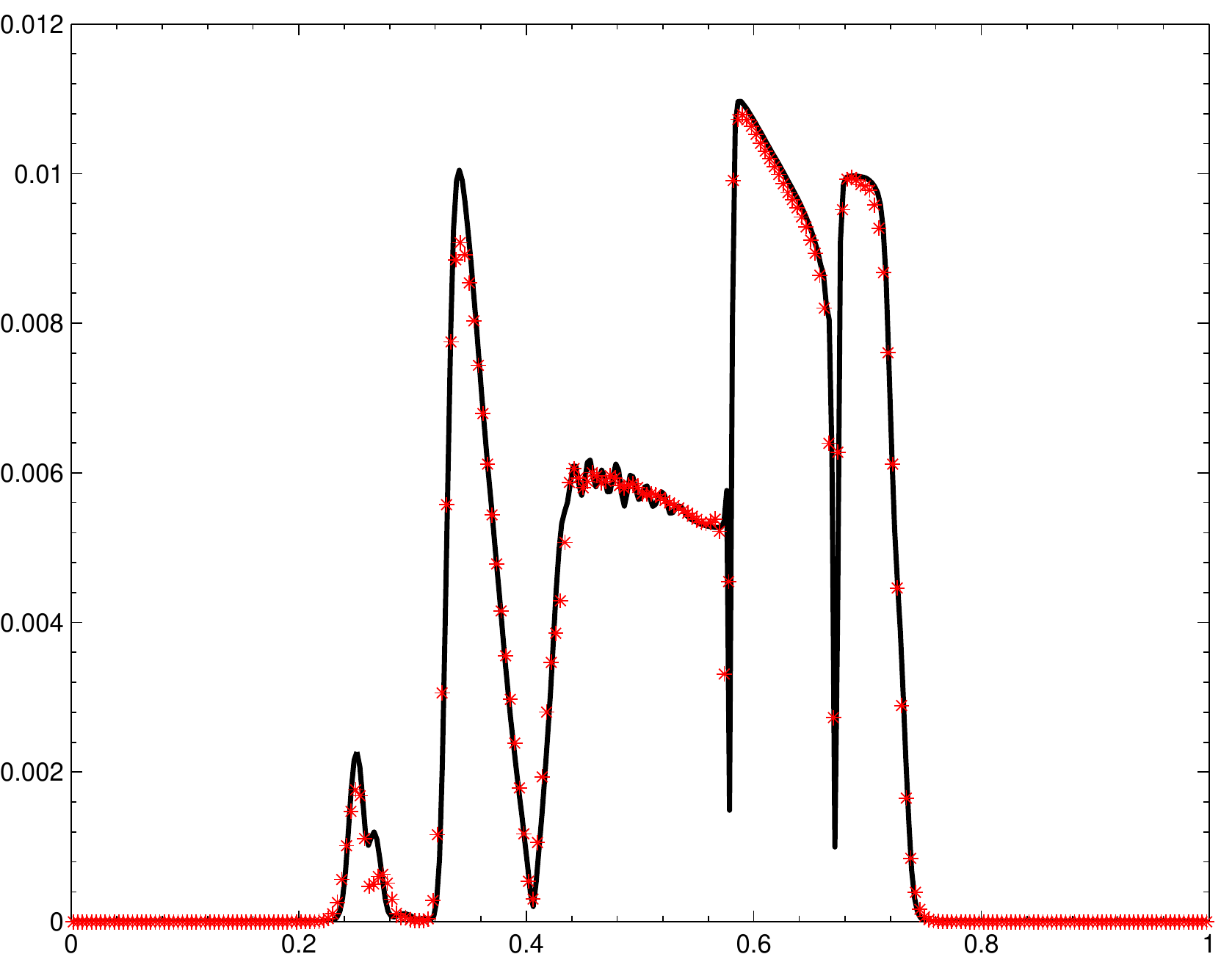}
  \caption{\small Same as Fig. \ref{fig:2DRP1a_Comparison}, except for the second configuration of Example \ref{example:RP2}.
 }
  \label{fig:2DRP2b_Comparison}
\end{figure}

\end{example}

\section{Conclusions}
\label{sec:conclude}

In this paper, an effective gPC stochastic Galerkin method for general
system of quasilinear hyperbolic conservation laws with uncertainty
is proposed. The advantage of the gPC-SG method is that its Galerkin
system of equations is proved to be symmetrically hyperbolic in 1D.
This allows one to employ a variety of mature deterministic schemes in
physical space and time. For 2D, the method can be readily adopted via
operator splitting. In this paper,
a high-order path-conservative finite volume WENO scheme in space, in
conjunction with a third-order explicit TVD Runge-Kutta method in time, is employed
and its properties studied.
Several examples for the Euler equations in 1D and 2D, exhibiting
complex structure in physical space, are presented to demonstrate the
accuracy and effectiveness of the proposed gPC-SG method.

\section*{Acknowledgements}

 KLW and HZT were partially supported by
the National Natural Science Foundation
of China (Nos.  91330205 \& 11421101).
DX was partially supported by AFOSR, DARPA, NSF.

\end{document}